\documentclass[12pt]{amsart}
\usepackage{amsmath}
\usepackage{amsfonts,amssymb}
\usepackage{euscript}
\usepackage{amsthm}
\usepackage[matrix,arrow,curve]{xy}
\usepackage{array}
\usepackage{stmaryrd}
\usepackage{graphicx}
\usepackage{tikz}

\numberwithin{equation}{section}
\theoremstyle{plain}
\newtheorem{theorem}{Theorem}[section]
\newtheorem{theoremintro}{Theorem}

\newtheorem{lemma}[theorem]{Lemma}
\newtheorem{prop}[theorem]{Proposition}
\newtheorem{corollary}[theorem]{Corollary}

\theoremstyle{definition}
\newtheorem{definition}[theorem]{Definition}
\newtheorem{remark}[theorem]{Remark}
\newtheorem{example}[theorem]{Example}
\newcommand{\s}{\sigma}

\newcommand{\N}{\mathbb N}
\newcommand{\Z}{\mathbb Z}

\renewcommand{\L}{\mathbb L}
\renewcommand{\P}{\mathbb P}
\newcommand{\X}{\mathbb X}
\newcommand{\Y}{\mathbb Y}
\renewcommand{\AA}{\mathcal A}

\newcommand{\DD}{\mathcal D}
\renewcommand{\SS}{\mathcal S}

\newcommand{\EE}{\mathcal E}

\newcommand{\GG}{\mathcal G}
\newcommand{\HH}{\mathcal H}

\newcommand{\LL}{\mathcal L}

\newcommand{\TT}{\mathcal T}

\newcommand{\UU}{\mathcal U}
\renewcommand{\O}{\mathcal O}
\renewcommand{\k}{\Bbbk}

\newcommand{\xra}{\xrightarrow}
\renewcommand{\le}{\leqslant}
\renewcommand{\ge}{\geqslant}

\newcommand{\bul}{\bullet}

\newcommand{\modd}{\mathrm{mod{-}}}
\newcommand{\moddo}{\mathrm{mod}_0{-}}
\newcommand{\Modd}{\mathrm{Mod{-}}}

\DeclareMathOperator{\Hom}{\textup{Hom}}

\DeclareMathOperator{\Ext}{\textup{Ext}}

\DeclareMathOperator{\End}{\mathrm{End}}

\DeclareMathOperator{\Ob}{\mathrm{Ob}}
 
\DeclareMathOperator{\coker}{\mathrm{coker}} 
\DeclareMathOperator{\rep}{\mathrm{rep}}
\DeclareMathOperator{\wdtrep}{\widetilde{\mathrm{rep}}}
\DeclareMathOperator{\im}{\mathrm{im}}

\DeclareMathOperator{\rank}{\mathrm{rank}}

\DeclareMathOperator{\coh}{\mathrm{coh}}

\DeclareMathOperator{\Perf}{\mathrm{Perf}}

\DeclareMathOperator{\id}{\mathrm{id}}

\DeclareMathOperator{\op}{\mathrm{op}}
\DeclareMathOperator{\Proj}{\mathrm{Proj}}
\DeclareMathOperator{\Supp}{\mathrm{Supp}}

\DeclareMathOperator{\Top}{\mathrm{Top}}
\DeclareMathOperator{\len}{\mathrm{length}}
\DeclareMathOperator{\add}{\mathrm{add}}
\DeclareMathOperator{\Ab}{\mathrm{Ab}}
\DeclareMathOperator{\rad}{\mathrm{rad}}

\begin{document}

% MSC2020: 16E35, 16G20, 14F08
% 18G80 Derived categories, triangulated categories
% 16E35 Derived categories and associative algebras
% 16G20 Representations of quivers and partially ordered sets
% 14F08 Derived categories of sheaves, dg categories, and related constructions in algebraic geometry
% 14H60  	Vector bundles on curves and their moduli 

\title[Thick subcategories on weighted projective curves]{Thick subcategories on weighted projective curves and nilpotent representations of quivers}

\author{Alexey ELAGIN}
\address{University of Sheffield, School of Mathematical and Physical Sciences, 
 The Hicks Building,
 Hounsfield Road,
 Sheffield, UK.
 S3 7RH}
\email{alexey.elagin@gmail.com}

\begin{abstract}
    We continue the study of thick triangulated subcategories, started by Valery Lunts and the author in ``Thick subcategories on curves'', and consider  thick  subcategories in the derived category of coherent sheaves on a weighted projective curve and the corresponding abelian thick subcategories. Our main result is that any thick subcategory on a weighted projective curve either is equivalent to the derived  category of nilpotent representations of some quiver (we call such categories quiver-like) or is the orthogonal subcategory to an exceptional collection of torsion sheaves (we call such subcategories big).  We examine the structure of thick subcategories: in particular, for weighted projective lines, we prove that any admissible subcategory is generated by an exceptional collection and any exceptional collection is a part of a full one. We show that the derived categories of weighted projective curves satisfy the Jordan--H\"older property and do not contain phantoms.  Finally, we extend and simplify results from loc.\,cit., providing sufficient criteria for a triangulated or abelian category to be quiver-like. 
\end{abstract}

\maketitle

\section{Introduction}

This paper is devoted to the classification problem of thick triangulated subcategories in a given triangulated category, which attracted a lot of attention in the last decades.
We continue the work started in our paper \cite{EL} with Valery Lunts, where thick subcategories in the derived category $D^b(\coh X)$ of coherent sheaves on a smooth projective curve $X$ were studied. Recall that a full subcategory $\TT\subset D^b(\coh X)$ is called \emph{thick} if $\TT$ is triangulated and closed under taking direct summands in $D^b(\coh X)$. 
\begin{definition}
Let $Q$ be a quiver and $\Modd\k Q$ be the category of right representations of $Q$ over a field $\k$. Denote by $D^b_0(\k Q)$ the triangulated subcategory in $D^b(\Modd \k Q)$, generated by the simple modules concentrated in the vertices of $Q$. We call triangulated categories, equivalent to a category of the form $D^b_0(\k Q)$, \emph{quiver-like}.
\end{definition}
The main result in \cite{EL} is the following
\begin{theorem}
\label{th_intro0}
Let $X$ be a smooth projective curve over an algebraically closed field. Then any finitely generated thick subcategory in $D^b(\coh X)$, different from $D^b(\coh X)$, is \emph{quiver-like}.
\end{theorem}

In the present paper we extend and generalise results of \cite{EL} in several directions.
First, we obtain a more transparent and general sufficient condition for a triangulated category to be quiver-like. More precisely, we prove
\begin{theoremintro}[See {Theorem~\ref{th_39}}]
\label{th_introA}
Let $\TT$ be an algebraic triangulated category, linear over a field $\k$. Assume that $\TT$ is classically generated by a set $\{t_i\}_{i\in I}$ of objects, and for all $i,j\in I$
\begin{equation}
\label{eq_vertlikeintro}
\begin{aligned}
    &\Hom^p(t_i,t_j)=0, &&\text{if $p\ne 0,1$,}\\
    &\Hom(t_i,t_j)=0,  && \text{if $i\ne j$,}\\
    &\Hom(t_i,t_i)=\k. &&
\end{aligned}
\end{equation}
Then $\TT$ is quiver-like and $t_i$-s correspond to simple modules.
\end{theoremintro}
Compared with \cite{EL}, we get rid of finiteness assumptions on the generating set, of assuming $\TT$ has a dg enhancement, and  of assuming that the dg endomorphism algebra is formal.
Instead, we establish equivalence of $\TT$ with the perfect derived category of certain $A_{\infty}$-category with the set of objects $\{t_i\}$ and the graded space of morphisms from $t_i$ to $t_j$ being $\oplus_p\Hom^p_{\TT}(t_i,t_j)$. Due to the grading constraints the $A_{\infty}$-structure must be trivial, yielding the equivalence of $\TT$
and the quiver-like category whose simple modules have the same $\Ext$-spaces as $t_i$-s have.  

Second, for the derived category of a hereditary abelian category $\AA$, there is a bijection between the thick triangulated subcategories in $D^b(\AA)$ and the \emph{thick} (that is, abelian exact extension-closed) subcategories in $\AA$. This is shown in~\cite{Bruning}; see our Proposition~\ref{prop_hersubcat2} for a precise statement. Therefore, statements about thick subcategories in $D^b(\AA)$ can easily be transformed into statements about thick subcategories in $\AA$ and vice versa.  We extend our description in terms of quivers from thick subcategories in $D^b(\AA)$ to the corresponding thick subcategories in $\AA$.
\begin{definition}
    Let $Q$ be a quiver and $\Modd\ k Q$ be the category of right representations of $Q$ over a field $\k$. Denote by $\moddo\k Q\subset \Modd\k Q$ the full subcategory of finite-dimensional nilpotent representations. We will call abelian categories, equivalent to a category of the form $\moddo\k Q$, \emph{quiver-like}.
\end{definition}

As a corollary of Theorem~\ref{th_introA}, we get
\begin{theoremintro}[See {Corollary~\ref{cor_40}}]
\label{th_introB}
Let $\AA$ be an abelian category linear over a field  $\k$, and $\{t_i\}_{i\in I}$ be a family of objects in $\AA$ satisfying conditions~\eqref{eq_vertlikeintro}. Let $\TT=\langle t_i\rangle_{i\in I}\subset D^b(\AA)$ be the thick subcategory generated by $t_i$-s, let $\SS=\TT\cap \AA$ be the corresponding thick subcategory in $\AA$. Then $\TT$ and $\SS$ are quiver-like, and the $t_i$-s correspond to simple modules.
\end{theoremintro}

As another application of Theorem~\ref{th_introA} we deduce a characterization of abelian quiver-like categories:
\begin{theoremintro}[See Theorem~\ref{th_39ab}]
\label{th_introC}
Let $\AA$ be an abelian category linear over a field  $\k$. Then $\AA$ is quiver-like if and only if $\AA$ is essentially small, all objects have finite length, one has $\End(S)= \k$ for any simple object $S\in \AA$ and $\Ext^p(S_1,S_2)=0$ for any $p\ge 2$ and simple objects $S_1,S_2\in\AA$.
\end{theoremintro}

Results similar to Theorems~\ref{th_introA},~\ref{th_introB},~\ref{th_introC} are not new and can be obtained by alternative methods. In particular, Theorem~\ref{th_introC} can be deduced, following ideas of Gabriel~\cite{Gabriel} and without passing to derived categories. One can use a characterisation of finite length abelian categories via pseudocompact modules over pseudocompact algebras from~\cite{Gabriel} and then relate the latter to finite-dimensional comodules over coalgebras by~\cite[Prop. 5.9]{Simson}. Given that the abelian category was hereditary, the coalgebra has to be the path coalgebra of a quiver by~\cite{Chin}, and then the category of finite-dimensional comodules is equivalent to the quiver-like category for the dual quiver by~\cite[Prop. 8.1(d)]{Simson}. For Theorem~\ref{th_introA}, one can construct a hereditary bounded t-structure on $\TT$, whose heart $\HH$ is extension-generated by objects $t_i$-s and is a finite length hereditary abelian category with $t_i$-s being simple objects
(see~\cite{AlNofayee} or~\cite[Prop. 5.4]{KoenigYang}). Then one can show that there is an equivalence $D^b(\HH)\to \TT$ using~\cite[Th. 2.3 and Th. 3.3]{ChenRingel}, see also~\cite[Th. 1.1]{Hubery} for the case where $\TT$ is not assumed to be algebraic. Hence Theorem~\ref{th_introA} would follow from Theorem~\ref{th_introC} applied to $\HH$.  We thank the referee for pointing out the approach to Theorems~\ref{th_introA}--\ref{th_introC} outlined above. We find it useful to provide in the present paper a new and probably  shorter way of proving these results. Our method is different in that it treats triangulated categories first and deduces the result about abelian categories as a corollary. It also avoids the machinery of topological rings and modules of~\cite{Gabriel} and uses $A_\infty$-categories as the main technical instrument.

Third (and most importantly), we apply the methods developed in \cite{EL} and in the first part of this work to the description of thick subcategories in $D^b(\coh \X)$, where $\X$ is a weighted projective curve. Geometrically, a weighted projective curve $\X=(X,w)$ is given by a smooth projective curve $X$ and a weight function $w\colon X\to \N$ such that $w(x)=1$ for all but a finite number of points $x\in X$. The points where $w>1$ are called \emph{orbifold} points. For the definition of the category $\coh\X$ of coherent sheaves see Section~\ref{section_wpccoh} or \cite{LenzingReiten}; see also \cite{GL} for the special case of weighted projective \emph{lines}. 

In the second part of this paper we give an analogue of Theorem~\ref{th_intro0} for a weighted projective curve $\X$. We demonstrate that \textbf{most} of the thick subcategories  in $D^b(\coh \X)$ are quiver-like. Exceptions are in a sense ``close to'' the whole category $D^b(\coh \X)$, they can be explicitly constructed and form a finite set (recall that the only exception in the case of a smooth projective curve is $D^b(\coh X)$ itself). The reason for  such non-quiver-like subcategories  to appear is the presence of exceptional torsion sheaves, which are concentrated in the orbifold points of $\X$. For any point $x\in \X$ of weight $r=w(x)$ the category $\coh_x\X$ of torsion sheaves on $\X$ supported at $x$ is equivalent to the category of finite-dimensional nilpotent representations of a cyclic quiver of length $r$. Such a category is called a 
\emph{tube}  of rank~$r$ and denoted by $\UU_r$. If $r\ge 2$ then $\coh_x\X$ has exceptional objects: for example, simple sheaves supported at orbifold points are exceptional.
\begin{definition}[See {Definition~\ref{def_curvelike}}]
We call a subcategory in $D^b(\coh\X)$ generated by an exceptional collection of torsion sheaves \emph{small} and the orthogonal subcategory to a small subcategory \emph{big}. We use the same terminology for the corresponding subcategories in $\coh\X$.     
\end{definition}

The main result of the second part of the paper is the following
\begin{theoremintro}[See {Theorem~\ref{theorem_main}}]
\label{th_introD}
    Let $\X$ be a weighted projective curve over an algebraically closed field $\k$. Let $\TT\subset D^b(\coh \X)$ be a thick subcategory, and
    $\SS=\TT\cap\coh\X$. Then $\TT$ and $\SS$ are big or $\TT$ and $\SS$ are quiver-like. 
\end{theoremintro}

Note that in contrast with Theorem~\ref{th_intro0} we do not assume that $\TT$ is finitely generated, and give a description for the abelian subcategory $\SS$ as well. Note also that two alternatives in Theorem~\ref{th_introD} are not exclusive: big triangulated  subcategories can be quiver-like. 

Let us explain briefly the idea of the proof of Theorem~\ref{th_introD}. We consider two cases: whether a thick subcategory $\TT\subset D^b(\coh\X)$ contains simultaneously a vector bundle and a sphere-like torsion sheaf (see Definition~\ref{def_exc}) or not. In the first case, one can show that  $\TT$ is big. In the second case, $\TT$ has a semi-orthogonal decomposition $\TT=\langle\TT_1,\TT_2\rangle$, where $\TT_1\cap\coh\X$ contains only torsion sheaves and $\TT_2\cap\coh\X$ only torsion-free sheaves. We show that the corresponding abelian categories $\SS_i=\TT_i\cap \coh\X$ are finite length categories: the length of a sheaf $F\in\SS_i$ as an object in $\SS_i$ is bounded by $\dim(H^0(F))$ for $\SS_1$ and by $\rank(F)$ for $\SS_2$. The families of simple objects in $\SS_1$ and~$\SS_2$, taken together, form a generating set for~$\TT$ satisfying assumptions of Theorem~\ref{th_introA}. By Theorem~\ref{th_introA} we deduce that $\TT$ is quiver-like.

To clarify the notion of a big subcategory, we introduce the following definition.
\begin{definition}
    Let $\X=(X,w)$ be a weighted projective curve. We call a thick subcategory $\SS\subset \coh\X$  \emph{curve-like} if there exists a weighted projective curve $\X'=(X,w')$ and a fully faithful functor $\Phi\colon \coh\X'\to\coh\X$ preserving rank and support of coherent sheaves with  $\im\Phi=\SS$ (note that the underlying curves of $\X$ and $\X'$ are the same).
\end{definition}
It is known that, given a simple exceptional torsion sheaf $S$ on a weighed projective curve $\X$,
its orthogonal subcategory $\{F\in\coh\X\mid \Ext^i(S,F)=0\quad\text{for all $i$}\}$ is curve-like. As next theorem shows, big subcategories are close to curve-like. 
\begin{theoremintro}[See Proposition~{\ref{prop_bigbig}}]
\label{th_introE}
    For a weighted projective curve over  an algebraically closed field $\k$, the following conditions on a thick subcategory $\SS\subset \coh\X$ are equivalent:
    \begin{enumerate}
        \item $\SS=\langle E_1,\ldots, E_n\rangle^\perp\cap \coh\X$
        for some exceptional collection $E_1,\ldots,E_n$ of torsion sheaves (that is, $\SS$ is big);
        \item $\SS=\SS_1\times \SS_2$, where $\SS_1\subset\coh\X$ is curve-like and $\SS_2\subset\coh\X$ is small;
        \item $\SS$ contains a curve-like subcategory of $\coh\X$;
        \item $\SS$ contains a non-zero vector bundle and a sphere-like torsion sheaf;
        \item $\SS$ contains a non-zero vector bundle and $\SS$ is invariant under functors $c_x$ of ``twisting by line bundle $\O_{\X}(x)$'' for any point $x\in X$  (see Definition~\ref{def_cx}).
    \end{enumerate}
\end{theoremintro}

Note that exceptional collections in a   tube $\UU_r$ have been explicitly classified in  \cite{Di} and \cite{Krause_strings}. Therefore small (and hence big) subcategories in $D^b(\coh\X)$ can also be explicitly classified.

While preparing this manuscript, we discovered a recent paper~\cite{Cheng} by Yiyu Cheng studying thick subcategories on weighted projective \emph{lines} (i.e., \emph{rational} curves), whose results and methods overlap with ours, in particular, with Theorem~\ref{th_introE}. Its proof, given in \cite{Cheng}, is similar to ours. We decided to include our proof of Theorem~\ref{th_introE} for the sake of completeness and because we work in greater generality, allowing curves of higher genus. The main difference between our setting and the case of rational curves is in dealing with condition (5) since for a rational curve all twist functors $c_x, x\in X$ are isomorphic and in general they are not. 

We use Theorem~\ref{th_introD} to obtain information about the structure of thick subcategories on a weighted projective curve $\X$. Let us mention  some of these results. In Proposition~\ref{prop_admexc} we show that any admissible subcategory in $D^b(\coh \X)$ that is not big is generated by an exceptional collection. In particular (Corollary~\ref{cor_admexcP1}), if $\X$ is a weighted projective \textbf{line} then any admissible subcategory in $D^b(\coh \X)$ is generated by an exceptional collection, and any exceptional collection is a part of a full exceptional collection. On the contrary (Corollary~\ref{cor_nosod}), if $\X=(X,w)$ where $X\not\cong\P^1$, then any admissible subcategory in $D^b(\coh \X)$ is either big or small. One can use this to get an alternative proof of a well-known theorem  by Okawa about semi-orthogonal indecomposability of the derived category of a smooth projective curve, see Remark~\ref{rem_okawa}.

Recall that a triangulated category is said to satisfy the \emph{Jordan--H\"older property} if all maximal semi-orthogonal decompositions have the same collection of components up to equivalence. Also recall that a triangulated category $\TT\ne 0$ is called a \emph{phantom} if  $K_0(\TT)=0$.
Violation of the Jordan--H\"older property and existence of phantoms can be viewed as pathologies, however, they happen. Failure of the Jordan--H\"older property  for the  derived categories of some algebraic varieties was shown in~\cite{BohningBothmerSosna},~\cite{Kuznetsov_JordanHolder}, and the first phantoms were constructed in~\cite{GorchinskiyOrlov} and~\cite{Bohning+}. Some positive results in this direction are also available: in~\cite{Pirozhkov_delPezzo} the Jordan--H\"older property was proved  for the derived category of $\P^2$, and absence of phantoms was demonstrated for del Pezzo surfaces. These are essentially the only geometric examples of dimension $>1$ known so far where semi-orthogonal decompositions exist and there provably are no phantoms. 
In this direction, we contribute the following:
\begin{theoremintro}[Corollaries~\ref{cor_JH} and~\ref{cor_phantom}]
The derived categories of weighted projective curves have the Jordan--H\"older property and do not contain phantom subcategories.    
\end{theoremintro}

We call a thick subcategory $\TT\subset D^b(\coh\X)$ \emph{torsion} (resp. \emph{torsion-free}) if all sheaves in $\TT$ are torsion (resp. torsion-free). A torsion subcategory $\TT$ decomposes into the orthogonal direct sum 
$$\TT=\oplus_{x\in X} \TT_x,$$
where $\TT_x$ is a subcategory supported at point $x$. For any point $x$ thick subcategories in $D^b(\coh\X)$ supported at $x$ are in bijection with thick subcategories in $D^b(\UU_r)$, where $\UU_r$ is a   tube of rank $r=w(x)$. The latter have been classified in~\cite{Di} and~\cite{Krause_strings}. On the other hand, we demonstrate (Proposition~\ref{prop_ttf}) that any thick subcategory $\TT\subset D^b(\coh\X)$ that is not big has unique semi-orthogonal decomposition $\TT=\langle\TT_1,\TT_2\rangle$, where $\TT_1$ is torsion and $\TT_2$ is torsion-free. This indicates that description of torsion-free subcategories is the major part of the classification of thick subcategories on weighted projective curves.

The variety of different types of thick subcategories in $D^b(\coh \X)$ for a weighted projective line $\X$ of wild type is presented (in the form of an Euler diagram) on Figure~\ref{fig_1}.
\begin{figure}
    \label{fig_1}
   	\centering 
		\begin{tikzpicture}[>=stealth,scale=1]
            \draw[line width=1.5] (0,0) rectangle (12,6);
            \draw[line width=1.5] (4,0.1) rectangle (8,3);
            \draw[line width=1.5] (2,2) rectangle (10,8);
            \draw[line width=1.5] (2.1,4) rectangle (9.9,7.9);
            \draw[line width=1.5] (6,2.1) rectangle (9.8,7.8);
            \draw[line width=1.5] (6.1,4.1) rectangle (9.7,7.7);
            \draw[line width=1.5] (0.1,2.1) rectangle (3.9,3.9);
            \draw[line width=1.5] (2.1,2.2) rectangle (3.8,3.8);
            \draw  (1.1,0.3) node[font=\small] {Quiver-like};
            \draw  (2.5,7.6) node[font=\small] {Big};
            \draw  (7.1,7.4) node[font=\small] {Curve-like};
            \draw  (0.9,2.4) node[font=\small] {Torsion};
            \draw  (2.7,2.5) node[font=\small] {Small};
            \draw  (5.1,0.4) node[font=\small] {Torsion-free};
            \draw  (4.95,3.5) node[font=\small] {Admissible};
            \draw  (9.2,2.4) node[font=\small] {WPL};
        \end{tikzpicture}
    \caption{Variety of thick triangulated subcategories on a weighted projective line of wild type. Here WPL means ``equivalent to the derived category of a weighted projective line''.}
\end{figure}

\subsection{Open question, further directions}

We did not investigate here which quivers are realizable on curves, i.e., which quiver-like categories $\moddo \k Q$ appear as thick subcategories in categories of coherent sheaves on weighted projective curves. This question is apparently more difficult than for ordinary smooth curves and will probably be studied elsewhere. Recall that all module categories over wild algebras are ``equally complex'' in the sense that for any two such algebras $A,B$ there is a functor $\modd A\to \modd B$ which is injective  on isomorphism classes of indecomposables. In general, however, this functor cannot be fully faithful with thick image, and $\modd A$ is not equivalent to  a thick subcategory of $\modd B$. Likewise, one should  not expect that any wild category $\moddo \k Q$ can be  embedded into any wild category $\coh\X$.

It seems that quiver-like categories are interesting on their own as natural generalisations of hereditary module categories, and deserve some attention.

For example, we do not know what the 
Auslander -- Reiten quivers of $\moddo \k Q$ and $D^b_0(\k Q)$ look like in general.  For module categories of finite acyclic quivers it is classically known that $D^b(\modd \k Q)\cong D^b(\modd\k Q')$ if and only if $Q$  is obtained from $Q'$ by a finite number of reflections. Also, in \cite{MiyachiYekutieli} the group of autoequivalences of $D^b(\modd \k Q)$ is described: it is the semi-direct product of the automorphism group of the Auslander -- Reiten quiver   of $D^b(\modd \k Q)$ and the product $\prod_{i,j\in Q_0}GL_{d_{ij}}(\k)$, where $d_{ij}$ is the number of arrows from $i$ to $j$ in $Q$.      We are not aware of analogues of these results for triangulated quiver-like categories if quivers can have cycles or be infinite.

\subsection{Outline}
The paper is organised as follows. In Section~\ref{section_background} we recall necessary background on abelian, triangulated, dg and $A_\infty$-categories. In Section~\ref{section_hereditary} we gather  less standard  information about abelian hereditary categories, their derived categories, and thick subcategories in these. In Section~\ref{section_QL} we introduce quiver-like categories and generalise results from~\cite{EL} providing sufficient conditions for categories to be quiver-like. Theorems~\ref{th_introA}, \ref{th_introB}, and \ref{th_introC} are proved here. We also describe proper, finitely generated, strongly finitely generated quiver-like categories, and those having a Serre functor, see Proposition~\ref{prop_QLgeneral}. In Section~\ref{section_linestubes} we recall standard results about tubes and  representations of linear $A_n$-quivers, in particular, we describe their thick subcategories as direct sums of some categories of the same form. In Section~\ref{section_wpc} we give definitions and basic facts about weighted projective curves and their categories of coherent sheaves. In Section~\ref{section_bigsmall} we introduce big, small, and curve-like subcategories in coherent sheaves on weighted projective curves, and give different characterisations of big subcategories from Theorem~\ref{th_introE}. Results of this section are similar to those from~\cite{Cheng} but are obtained in greater generality. In Section~\ref{section_main} we prove our main result, Theorem~\ref{th_introD}. Further we discuss the structure and the variety of thick subcategories on weighted projective lines.
Finally, in Section~\ref{section_examples} we provide some examples of thick quiver-like subcategories on weighted projective lines and compute corresponding quivers.

\subsection{Acknowledgements}

This paper grows out from the study of thick subcategories on smooth curves carried out in collaboration with Valery Lunts, to whom I am much indebted. 
Major part of this work was done in the inspiring environments of the IHES and the University of Edinburgh, and I am grateful to Emmanuel Ullmo, Maxim Kontsevich, Mikhail Tsfasman, Vanya Cheltsov, and 
Sasha Shapiro for their hospitality and help during that period. 
I thank Rudradip Biswas, Nathan Broomhead, Martin Gallauer, Edmund Heng, Yuki Hirano, Daigo Ito, and Dmitri Kaledin for their interest in this study and for useful discussions. 
Special thanks go to Timothy Logvinenko for his help with $A_\infty$-categories. I am thankful to an anonymous referee for careful reading of the manuscript and for their comments. 
Finally, I thank the Isaac Newton Institute, the London Mathematical Society, and the UKRI Horizon Europe guarantee award 
`Motivic invariants and birational geometry of simple normal crossing degenerations' EP/Z000955/1 for their financial support.

\section{Background, conventions, notation}
\label{section_background}
We work over a fixed field $\k$. Starting from Section~\ref{section_wpc} we assume that $\k$ 
is algebraically closed. For a $\k$-vector space $V$, we denote its dual by $V^*$. An additive category $A$ is \emph{$\k$-linear} if all its $\Hom$ groups are $\k$-vector spaces and composition maps are $\k$-bilinear. A $\k$-linear additive category is \emph{$\Hom$-finite} if all $\Hom$ spaces are finite-dimensional over $\k$.

\subsection{Background on abelian categories}
An object in an abelian category is \emph{simple} if it has no non-trivial subobjects. An object $F$ is said to have \emph{finite length} if it has a finite filtration with simple quotients. The number of such quotient does not depend on filtration and is called the \emph{length} of $F$. 
An abelian category is a \emph{finite length category} if any its object has finite length. 

An abelian category $\AA$ is \emph{connected} if it has no non-trivial decompositions  into a direct product $\AA=\AA_1\times\AA_2$.

An object $A$ of an abelian  category $\AA$ is \emph{uniserial} if it has unique filtration $0=A_0\subset A_1\subset\ldots\subset A_m=A$ with simple quotients. All subobjects   of a  uniserial  object are terms of this filtration. An abelian category is \emph{uniserial} if all its indecomposable objects are uniserial. Any object in a uniserial category is uniquely determined by its length $m$ and its top simple quotient $A_m/A_{m-1}$.

The \emph{radical} $\mathrm{rad}(A)$ of an object $A$ in a finite length abelian category is the intersection of all maximal subobjects. Equivalently, radical of $A$ is the kernel of the universal semi-simple quotient of $A$. The latter is called the \emph{top} of $A$ and denoted $\mathrm{top}(A)$.

For an abelian category $\AA$, we denote by $\Gamma_0(\AA)$ a set of representatives of isomorphism classes of simple objects in $\AA$. 

\subsection{Background on triangulated and derived categories}

We refer to \cite{BK}, \cite{BvdB}, \cite{GelfandManin}, \cite{Huyb}, \cite{Neeman}   for the definitions and basic concepts on  triangulated and derived categories. We denote shift functor by $[\:]$ and call distinguished triangles \emph{exact}.

For objects $X,Y$ in a triangulated category $\DD$ we denote $\Hom^i(X,Y):=\Hom(X,Y[i])$ and 
$$\Hom^\bul(X,Y):=\oplus_{i\in\Z}\Hom^i(X,Y),$$
this is a graded abelian group. A triangulated $\k$-linear category $\DD$ is \emph{proper} if  for any $X,Y\in \DD$ the $\k$-vector space $\Hom^\bul(X,Y)$ is finite-dimensional. 

Let $\GG$ be a subcategory (or a collection of objects, or a single object) in a triangulated category $\DD$. We denote by $[\GG]$ the smallest strict  full triangulated subcategory in $\DD$ that contains $\GG$ and call it \emph{the triangulated subcategory generated by} $\GG$. We denote by $\langle\GG\rangle$ the smallest strict idempotent closed full triangulated subcategory in $\DD$ that contains $\GG$ and call it \emph{the thick subcategory generated by} $\GG$. 

Subcategories $[\GG],\langle\GG\rangle\subset \DD$ can be described  constructively as follows. Denote by $[\GG]_0$ the full subcategory in $\DD$ whose objects are  finite direct sums of shifts of objects in $\GG$. For $n\ge 1$, denote by $[\GG]_n$ the full subcategory in $\DD$ whose objects $X$ fit into an exact triangle 
$X_0\to X\to X_{n-1}\to X_0[1]$, where $X_i\in [\GG]_i$.  Let $\langle\GG\rangle_n$ be the idempotent closure of $[\GG]_n$. Then $[\GG]=\cup_n [\GG]_n$  and $\langle\GG\rangle=\cup_n \langle\GG\rangle_n$.
We say that $\GG$ \emph{generates $\DD$ as a triangulated category} if $[\GG]=\DD$.

An object $G$ in $\DD$ is called a \emph{generator} of $\DD$ if $\langle G\rangle=\DD$, and is called a \emph{strong generator} of $\DD$ if $\langle G\rangle_n=\DD$ for some $n$.

A \emph{Serre functor} on a triangulated $\Hom$-finite $\k$-linear category $\DD$ is an autoequivalence $S\colon \DD\to \DD$ 
such that there exist  natural isomorphisms 
$$\Hom(X,Y)\cong \Hom(Y,S(X))^*$$
for all $X,Y\in \DD$.  Such functor, if exists, is exact and unique up to an isomorphism.

For a subcategory $\TT\subset \DD$ (or a family of objects, or for a single object) define its \emph{left} and \emph{right orthogonals} as full subcategories in $\DD$, given  respectively by 
\begin{align*}
    ^\perp \TT &:=\{X\in \DD\mid \Hom^\bul(X,T)=0\quad \text{for any $T\in\TT$}\},\\
    \TT^\perp &:=\{X\in \DD\mid \Hom^\bul(T,X)=0\quad \text{for any $T\in\TT$}\}.\\
\end{align*}
One says that a triangulated category $\DD$ has a \emph{semi-orthogonal decomposition} 
$$\DD=\langle\TT_1,\ldots,\TT_n\rangle$$
if  $\TT_1,\ldots,\TT_n$ are  full triangulated subcategories and  
\begin{enumerate}
    \item $\TT_i\subset \TT_j^\perp$ for all $1\le i<j\le n$, and
    \item $\DD$ is the smallest strict full triangulated subcategory of $\DD$ that contains all $\TT_1,\ldots,\TT_n$.
\end{enumerate}
In this case the following holds:
\begin{itemize}
    \item For any $X\in \DD$ there is unique diagram
    $$0=X_{n+1}\to X_n\to\ldots \to X_{2}\to X_1=X,$$
    such that $Y_i:=Cone(X_{i+1}\to X_i)\in\TT_i$ for all $i=1\ldots n$.
    \item There are well-defined exact functors $\pi_i\colon \DD\to\TT_i, \pi_i(X)=Y_i$, called \emph{projection functors}.
    \item $\pi_1\colon \DD\to\TT_1$  is left adjoint to the inclusion $\TT_1\to\DD$, and 
    $\pi_n\colon \DD\to\TT_n$  is right adjoint to the inclusion $\TT_n\to\DD$.
    \item One has 
    $$\TT_i=^\perp\langle\TT_1,\ldots,\TT_{i-1}\rangle\cap \langle\TT_{i+1},\ldots,\TT_n\rangle^\perp.$$
    \item Associativity: for any $1\le p\le q\le n$ there is a semi-orthogonal decomposition
    $$\DD=\langle\TT_1,\ldots,\TT_{p-1},\langle\TT_p,\ldots,\TT_q\rangle,\TT_{q+1},\ldots,\TT_n\rangle.$$
\end{itemize}
A subcategory $\TT\subset \DD$ is called \emph{left} (resp. \emph{right}) \emph{admissible} if the inclusion functor $\TT\to \DD$ has a left (resp. right) adjoint functor. A triangulated subcategory $\TT$ of a triangulated category $\DD$ is left admissible if and only if it appears as a left component in some semi-orthogonal decomposition $\DD=\langle\TT,\TT'\rangle$,
and if and only if there is semi-orthogonal decomposition $\DD=\langle\TT,^\perp \TT\rangle$. Similar statement holds for right admissible subcategories.

A subcategory $\TT\subset \DD$ is called \emph{admissible} if it is both left and right admissible.

\begin{definition}
\label{def_exc}
Let $\DD$ be a triangulated $\k$-linear category. An object $E\in\DD$  is \emph{exceptional} if 
$$\Hom(E,E)=\k, \quad \Hom^i(E,E)=0\quad\text{for $i\ne 0$}.$$
An object $E\in\DD$ is \emph{sphere-like} (more precisely, $1$-sphere-like, but we will not encounter any other sphere-like objects) if 
$$\Hom(E,E)=\k, \quad \Hom^1(E,E)\cong \k, \quad \Hom^i(E,E)=0\quad\text{for $i\ne 0,1$}.$$
A collection of objects $E_1,\ldots,E_n\in\DD$ is \emph{exceptional} if 
\begin{itemize}
    \item any $E_i$ is exceptional, and
    \item $\Hom^\bul(E_q,E_p)=0\quad \text{for all $1\le p<q\le n$}$.
\end{itemize}
\end{definition}

If $\DD$ is proper then the triangulated subcategory in $\DD$ generated by an exceptional collection $E_1,\ldots,E_n$ is admissible, we denote this subcategory 
$$\langle E_1,\ldots,E_n\rangle.$$ 
In particular, one has semi-orthogonal decompositions
$$\langle E_1,\ldots,E_n\rangle=\langle\langle E_1\rangle,\ldots,\langle E_n\rangle\rangle$$
and
$$\DD=\langle\langle E_1,\ldots,E_n\rangle^\perp,\langle E_1,\ldots,E_n\rangle\rangle=\langle \langle E_1,\ldots,E_n\rangle, ^\perp\langle E_1,\ldots,E_n\rangle\rangle.$$

\medskip
A triangulated category is called \emph{algebraic} if it is the stable category of a Frobenius exact category (see  \cite{Kr} or \cite{LH}). Algebraic triangulated categories include homotopy and derived  categories of abelian  categories, and also their triangulated subcategories and Verdier localizations. All triangulated categories that we come across in this paper are algebraic.

\medskip
For an abelian category $\AA$ we denote by $D^b(\AA)$ its bounded derived category. We will tacitly identify $\AA$ with the image of the fully faithful functor $\AA\to D^b(\AA)$. For a  subcategory
$\TT\subset D^b(\AA)$ we denote 
$$\TT_{\Ab}:=\TT\cap \AA,$$
this is a subcategory in $\AA$. In particular, for a subcategory or an object $\SS\subset \AA$ we have
$$\SS^\perp_{\Ab}:=\SS^\perp\cap \AA=\{A\in \AA\mid \Ext^i(S,A)=0\quad\text{for all $i$ and $S\in\SS$}\}$$
and similarly for $^\perp\SS_{\Ab}$.

We will speak about exceptional and sphere-like objects and collections in an abelian category $\AA$, meaning that they are  such in $D^b(\AA)$.

Recall that the Grothendieck group $K_0(\AA)$ of an abelian category $\AA$ is the abelian group generated by the isomorphism classes $[A]$ of objects $A\in\AA$ and relations  $-[A]+[B]-[C]=0$ for all exact triples
$0\to A\to B\to C\to 0$ in $\AA$. Similarly, the Grothendieck group $K_0(\DD)$ of a triangulated category $\DD$ is generated by the isomorphism classes $[X]$ of objects $X\in\DD$ and relations  $[X]+[X[1]]=0$ and $-[X]+[Y]-[Z]=0$ for all exact triangles
$X\to Y\to Z\to X[1]$ in $\DD$. There is a natural isomorphism $K_0(\AA)\to K_0(D^b(\AA))$, sending $[A]$ to $[A]$ for any $A\in\AA$. Assume that $\DD$ is $\k$-linear and  proper, then there is bilinear Euler form $\chi$ on $K_0(\DD)$, defined by 
$$\chi([X],[Y])=\sum_{i\in\Z}(-1)^i\dim\Hom^i(X,Y)$$
for all $X,Y\in\DD$.

\subsection{Background on dg and $A_{\infty}$-categories}

We refer to \cite{BondalLarsenLunts}, \cite{Keller_DerivingDG}, \cite{Keller_Ainfinity}, \cite{LH} for main definitions and constructions related to differential graded (=dg) categories and 
$A_{\infty}$-categories. We will be very brief here as we use these machinery only in the proof of Theorem~\ref{th_39}.

By definition, a differential graded (dg) category is a $\k$-linear category, whose $\Hom$ spaces carry a structure of differential complexes of $\k$-vector spaces, composition maps are homogeneous and satisfy graded  Leibniz rule:  
$$d(a\cdot b)=d(a)\cdot b + (-1)^p a\cdot d(b)$$
for any $X,Y,Z\in \Ob(\AA)$, $a\in \Hom^p(Y,Z), b\in \Hom^q(X,Y)$.
Given a dg category, one defines its \emph{derived category} $D(\AA)$ as the localisation of the homotopy category of right dg $\AA$-modules by the class of quasi-isomorphisms, see \cite{Keller_DerivingDG} for details. Then $D(\AA)$  is a triangulated category, it contains a family of \emph{representable modules} $h^X$, $X\in\Ob(\AA)$. \emph{Perfect derived category} $\Perf(\AA)$ of $\AA$ is defined as the smallest triangulated subcategory in $D(\AA)$ containing all modules $h^X$ for $X\in\Ob(\AA)$ and closed under taking direct summands.

\medskip
By definition, an $A_{\infty}$-category $\AA$ over $\k$ is given by 
\begin{itemize}
    \item a family of objects $\Ob(\AA)$, 
    \item a family of $\Z$-graded vector spaces $\Hom(X,Y)$ for all $X,Y\in Ob(\AA)$,
    \item for any $n\ge 1$ a $\k$-linear  homogeneous of degree $2-n$ homomorphism
$$m_n\colon \Hom(X_{n-1},X_n)\otimes \ldots\otimes \Hom(X_1,X_2)\otimes \Hom(X_0,X_1)\to \Hom(X_{0},X_n),$$
which are subject to relations
\begin{equation}
\label{eq_relainfty}
    \sum_{n=r+s+t, r,t\ge 0, s\ge 1} (-1)^{r+st}m_{r+1+t}\circ (1^{\otimes r}\otimes m_s\otimes 1^{\otimes t})=0
\end{equation}
for any $n\ge 1$.
\end{itemize}

Parallel to the case of dg categories, for an $A_{\infty}$-category one defines its \emph{derived category} $D^\infty(\AA)$ and \emph{perfect derived category} $\Perf^\infty(\AA)\subset D^\infty(\AA)$ (see~\cite{Keller_Ainfinity}), these categories are triangulated. The latter is generated as an idempotent-closed triangulated subcategory in $D^\infty(\AA)$ by the family of  \emph{representable modules} $h^X, X\in\Ob(\AA)$.

An $A_{\infty}$-category is called \emph{strictly unital} if there are elements $1_X\in \Hom(X,X)$ of degree~$0$ for all $X\in\Ob(\AA)$ such that 
$$m_n(a_n\otimes\ldots\otimes a_1)=0$$
as soon as at least one of $a_1,\ldots,a_n$ is $1_X$ for some $X$, and 
$$m_2(a\otimes 1_X)=a=m_2(1_Y\otimes a)$$
for any $a\in\Hom(X,Y)$.

An $A_{\infty}$-category is called \emph{minimal} if $m_1=0$. 

Assume $\AA$ is a strictly unital $A_{\infty}$-category
and all $m_n, n\ge 3$ vanish.
Then relations~\eqref{eq_relainfty} take the form
$$m_1^2=0, \quad m_1m_2=m_2(1\otimes m_1+m_1\otimes 1), \quad m_2(m_2\otimes 1)=m_2(1\otimes m_2),$$
so that  $m_2$ is associative. Hence $m_2$ makes $\AA$ a category (a general $A_{\infty}$-category is not a category!), $m_1$ makes any $\Hom(X,Y)$ a differential complex, and composition maps
$\Hom(Y,Z)\otimes \Hom(X,Y)\to\Hom(X,Z)$ are compatible with the differentials.
Therefore,~$\AA$ is in fact a dg category.

\section{Abelian hereditary categories, their derived categories, and their subcategories}
\label{section_hereditary}

In this paper we deal with hereditary abelian categories and their derived categories, and study thick subcategories in them. Here we give necessary definitions and explain that, given a hereditary abelian category $\AA$, thick subcategories in $\AA$ and in $D^b(\AA)$ are in a very natural bijection.

\begin{definition}
An abelian category $\AA$ is called \emph{hereditary} if $\Ext_\AA^i(X,Y)=0$ for all $i\ge 2$, $X,Y\in\AA$.
\end{definition}

There are several reasonable properties for a subcategory in an abelian category. 

\begin{definition}[See {\cite[Sect. 4.4]{Kr}}, \cite{Di}, \cite{Hovey}, \cite{Cheng}]
Let $\SS\subset \AA$ be a full subcategory of an abelian  category. Then
\begin{itemize}
    \item $\SS$ is called  \emph{wide}  if $\SS$ is closed under taking kernels, cokernels, and extensions,
    \item $\SS$ is called \emph{thick} if $\SS$ is closed under taking direct summands and \emph{2-of-the-3 condition} is satisfied: if in an exact sequence $0\to X\to Y\to Z\to 0$ in $\AA$ two of $X,Y,Z$ are in $\SS$, then the third one is also in $\SS$,
    \item $\SS$ is called \emph{Serre} if $\SS$ is closed under taking subobjects, quotient objects, and extensions.
\end{itemize}
\end{definition}

Note that a wide subcategory of an abelian category is abelian, and the inclusion functor is exact. 
Also note that the terminology slightly varies across the literature.

\begin{prop}
Any wide subcategory of an abelian category is thick. Moreover, if $\AA$ is hereditary then any thick subcategory in $\AA$ is wide.      
\end{prop}
\begin{proof}
    The first is obvious, for the second see \cite[Th. 3.3.1]{Di} or \cite[Rem. 4.4.16]{Kr}.
\end{proof}

In this paper we will deal with wide (= thick) subcategories of hereditary abelian categories and will call them thick.
We will make use of the following important 
\begin{prop}[See {\cite[Lemma A.1, Prop. A.2]{RvdB}}]
\label{prop_hersubcat1}
\begin{enumerate}
    \item An abelian category $\AA$ is hereditary if and only if the functor $\Ext^1_\AA(X,-)$ is right exact for any $X\in\AA$.
    \item Let $\AA$ be hereditary and $\SS\subset\AA$ be a thick subcategory. Then $\SS$ is also a hereditary abelian category, and the inclusion functor is exact.  
\end{enumerate}
\end{prop}

We will also deal with subcategories of derived categories. Recall 
\begin{definition}
A full subcategory $\TT$ in a triangulated category is \emph{thick} if $\TT$ is closed under taking shifts, cones, and direct summands. We will also assume that $\TT$ is closed under isomorphisms.
\end{definition}

Derived category $D^b(\AA)$ of a hereditary abelian category $\AA$ has a simple structure: recall well-known
\begin{prop}[See, for example, {\cite[Prop. 4.4.15]{Kr}}]
Let $\AA$ be abelian hereditary category. Then  any object $X\in D^b(\AA)$ is isomorphic to the direct sum of its shifted cohomology:
$$X\cong \oplus_iH^i(X)[-i].$$ 
In particular, any indecomposable object in $D^b(\AA)$ is a shift of some indecomposable object in $\AA$. 
\end{prop}
This allows to relate thick subcategories in $D^b(\AA)$ and in $\AA$.

\begin{prop}
\label{prop_hersubcat2}
Let $\AA$ be an abelian hereditary category. 
\begin{enumerate}
    \item Then there is a bijection between thick subcategories in $\AA$ and in $D^b(\AA)$, given by assignments
\begin{align*}
    \AA\supset \SS&\mapsto \{X\mid H^i(X)\in \SS\quad\text{for all $i$}\}=\langle \SS\rangle\subset D^b(\AA),\\
    \AA\supset \AA\cap \TT &\mapsfrom \TT\subset D^b(\AA).
\end{align*}
    \item Moreover, let $\SS\subset \AA$ and $\TT\subset D^b(\AA)$ be corresponding thick subcategories. Then 
    the derived functor  $D^b(\SS)\to D^b(\AA)$ of the inclusion is fully faithful with the image~$\TT$.  
\end{enumerate}
\end{prop}
\begin{proof}
    This is \cite[Th. 5.1]{Bruning} or \cite[Prop. 4.4.17]{Kr} for (1) and Prop.~\ref{prop_hersubcat1} combined with \cite[Prop. 4.4.17]{Kr} for~(2). 
\end{proof}

\section{Quiver-like categories}
\label{section_QL}

Here we introduce one of the main characters of this paper, quiver-like categories. 
Most results of this section can be found in~\cite{EL}. However, now  we present  a different point of view and work in greater generality. In particular, we do not assume quivers to be finite, in contrast to loc. cit., do not rely on enhancements  of triangulated categories, and remove formality assumptions.

\medskip
A \emph{quiver} $Q=(Q_0,Q_1)$ is given by a set $Q_0$ of vertices, a set $Q_1$ of arrows, and two maps~$s$~(source) and $t$ (target) from $Q_1$ to $Q_0$. Let $\k Q$ denote the path algebra of $Q$ over a field~$\k$. Let $\rad(\k Q)\subset \k Q$ be the two-sided ideal, spanned by all paths of positive length.  A \emph{(right) representation} of $Q$ over $\k$ is a collection of $\k$-vector spaces $V_i$ for $i\in Q_0$, and homomorphisms $V_{t(a)}\to V_{s(a)}$ for $a\in Q_1$. 
Any representation $(V_i)_{i\in Q_0}$ of $Q$ can be viewed as a (right) module $\oplus_{i\in Q_0}V_i$ over $\k Q$. 
A representation $(V_i)_{i\in Q_0}$ is \emph{finite-dimensional} if the vector space $\oplus_{i\in Q_0}V_i$ is finite-dimensional. A representation $(V_i)_{i\in Q_0}$ is \emph{radical-nilpotent} if  $(\oplus_{i\in Q_0}V_i)\cdot \rad(\k Q)^N=0$ for some $N\in \N$.
Let $\Modd \k Q$ denote the category of right representations of $Q$ over $\k$ (we will also call them modules), it is a hereditary abelian category. 
\begin{definition}
    Let $\moddo\k Q\subset \Modd \k Q$ be the full subcategory of finite-dimensional radical-nilpotent representations, this is a thick and wide (and even Serre) subcategory. Let $D^b_0(\k Q)\subset D^b(\Modd \k Q)$ denote the full subcategory of complexes with cohomology  in $\moddo \k Q$, this is a thick subcategory.

\end{definition}
Denote by $s_i$ the simple module concentrated in vertex $i\in Q_0$. It is easy to see that $\moddo\k Q$ consists precisely of modules having a finite filtration with quotients $s_i$, and 
$D^b_0(\k Q)$ is generated by the modules $s_i$ as triangulated category.
\begin{definition}
An abelian category $\AA$ is \emph{quiver-like} if there is an additive equivalence  $\AA\to \moddo\k Q$ for some quiver $Q$.
A triangulated category $\TT$ is called \emph{quiver-like} if there is an exact  equivalence $\TT\to D^b_0(\k Q)$ for some quiver $Q$.
\end{definition}

\begin{remark}
\label{rem_uniquemoddo}
Quiver $Q$ can be reconstructed from the associated abelian category $\moddo\k Q$  as the $\Ext$-quiver of the collection of simple objects, see Definition~\ref{def_extquiver} below. 

On the contrary, there are non-trivial equivalences between triangulated quiver-like categories. For example, 
$D^b_0(\k Q)\cong D^b_0(\k Q')$ if $Q$ is finite acyclic and $Q'$ is obtained from $Q$ by a sequence of reflections (known since \cite{Happel}, see also \cite{Parthasarathy}). We do not know when, for general quivers $Q,Q'$, the categories $D^b_0(\k Q)$ and $D^b_0(\k Q')$ are equivalent.
\end{remark}

\begin{prop}
\label{prop_moddo}
The category $\moddo\k Q$ is hereditary, and the natural functor 
$$D^b(\moddo\k Q)\to D^b_0(\k Q)$$ is an equivalence.    
\end{prop}
\begin{proof}
    It follows from Propositions~\ref{prop_hersubcat1} and~\ref{prop_hersubcat2}, applied to the hereditary category $\Modd\k Q$.
\end{proof}

We have trivial but pleasant
\begin{prop}[{\cite[Corollary 3.4]{EL}}]
\label{prop_pleasant}
Let $\TT$ be a quiver-like triangulated category with an equivalence
$\Phi\colon  D^b_0(\k Q)\to \TT$. Put $t_i := \Phi(s_i)$. Then we have the following.
\begin{enumerate}
\item  For any indecomposable object $X \in  \TT$ there exists $d\in\Z$ and a diagram
$$0=X_0\to X_1\to\ldots\to X_m=X[d]$$
such that $ Cone(X_{j-1}\to X_j)\cong t_{i_j}$ for each $j = 1, \ldots, m$ and some $i_j\in Q_0$.
\item $K_0(\TT) \cong \oplus_{i\in Q_0} \Z\cdot [t_i]$.
\item $\TT$ is generated by $t_i, i\in Q_0$, as a triangulated category.
\end{enumerate}
\end{prop}

Next we provide some general properties of quiver-like categories.
Recall that a quiver is \emph{acyclic} if it has no oriented cycles.

\begin{prop}
\label{prop_QLgeneral}
Let $Q=(Q_0,Q_1)$ be a quiver and $\TT=D^b_0(\k Q)$ be the corresponding quiver-like category. Assume also that $Q$ is connected. Then
\begin{enumerate}
\item $\TT$ has a generator $\Longleftrightarrow$ $Q_0$ is finite. If this is the case, then $G:=\oplus_{i\in Q_0} s_i$  is also a generator.
\item $\TT$ is proper $\Longleftrightarrow$ any two vertices are connected by finitely many arrows.  
\item $\TT$ has a strong generator $\Longleftrightarrow$ $Q_0$ is finite and $Q$ is acyclic.
\item $\TT$ has a Serre functor  $\Longleftrightarrow$ one of the following holds:
\begin{enumerate}
    \item for any vertex $i$ in $Q$ there are only finitely many paths containing  $i$,
    \item $Q$ is a cycle $Z_n$ of length $n\ge 1$,
    \item $Q$ is $A_{\infty,\infty}$: 
    $$\ldots \to \bul\to\bul\to\bul\to \ldots$$
\end{enumerate}
\end{enumerate}
\end{prop}
\begin{proof}
    
(1) and (2) are easy and left to the reader. 

For (3), assume $Q$ is acyclic with finitely many vertices. Then $\moddo\k Q$ is the category of all finite-dimensional modules over path algebra $\k Q$, let $R=\rad(\k Q)$. Let $n$ be the maximal length of a path in $Q$, then $R^{n+1}=0$. Put $G:=\oplus_{i\in Q_0}s_i$, we claim that  $[G]_n=D^b_0(\k Q)$. Indeed, for a bounded complex $M$ of $\k Q$-modules, consider its  filtration by subcomplexes
$$0=M \cdot R^{n+1}\subset M \cdot R^{n}\subset \ldots \subset M \cdot R\subset M.$$
Any quotient $F_j:=M \cdot R^{j}/M \cdot R^{j+1}$ is a bounded complex, where each term  is a finite direct sum of some modules $s_i, i\in I$. The category $\add(G)$
 of such direct sums is semi-simple, therefore $F_j$ is quasi-isomorphic to the direct sum of its cohomology modules, which are in $\add(G)$. Hence $F_j\in [G]_0$. 
It follows that $M\in [G]_n$, and $G$ is a strong generator of $D^b_0(\k Q)$.

Contrary, assume $D^b_0(\k Q)$ has a strong generator, then any generator is  strong  (see~\cite[Sect. 3.1]{Rouquier}). From (1) we deduce that $Q_0$ is finite and $G$ is a strong generator: $\langle G\rangle_n=D^b_0(\k Q)$ for some $n$. Note that $G\cdot R=0$. Using standard technique, one can prove by induction that $M\cdot R^{j+1}=0$ for any $M\in \moddo\k Q\cap \langle G\rangle_j$. Therefore,  $M\cdot R^{n+1}=0$ for all $M\in \moddo\k Q$. It follows then that the length of paths in $Q$ is bounded by $n$, in particular, $Q$ cannot have cycles.

(4) is based on the classification of Noetherian hereditary abelian categories with a Serre functor, given in~\cite{RvdB}. If (a) holds, then  all indecomposable   
projective and injective $Q$-modules are in $\moddo \k Q$, and $\moddo \k Q$  has enough projectives and injectives. Then the derived Nakayama functor $D^b(\moddo k Q)\to D^b(\moddo \k Q)$ is well-defined, is an equivalence, and serves as a Serre functor on $D^b(\moddo\k Q)=D^b_0(\k Q)$. If (b) or (c) holds, $D^b(\k Q)$ has a Serre functor by~\cite[Th. B]{RvdB}.

The other way, recall that $\AA=\moddo\k Q$ is a connected Noetherian hereditary abelian category of finite length. Assume that $D^b(\AA)$ has a Serre functor (in particular, that $\AA$ is $\Ext$-finite), then one of the following holds by~\cite[Th. B]{RvdB}:
\begin{enumerate}
    \item $\AA$ has no non-zero projective objects. Then $\AA\cong \moddo \k Z_n$ or $\AA\cong \moddo \k A_{\infty,\infty}$ (this is case (a) in \cite[Th. B]{RvdB} or \cite[Th. III.1.1]{RvdB}).
    \item $\AA$ has a non-zero projective object. Then by \cite[Th. II.4.9]{RvdB} $\AA$ is equivalent to category $\wdtrep(\Gamma)$, defined in \cite[Sect. II]{RvdB}, where $\Gamma$ is a quiver obtained in the following way. Start with a quiver $\Gamma'$ such that any vertex is contained in only finitely many paths. Then, to each vertex $\gamma$ in $\Gamma'$, attach finitely many (possibly zero) infinite rays of the form $A_{\infty,0}$: 
    \begin{equation}
    \label{eq_ray}
        \ldots \to 3\to 2\to 1\to 0,
    \end{equation}
    by gluing their terminal vertex $0$ to $\gamma$.
\end{enumerate}
In case (1) it remains to note that equivalence $\moddo \k Q\cong \moddo \k Q'$ implies $Q=Q'$, see Remark~\ref{rem_uniquemoddo}.
In case (2) we prove that $\Gamma=\Gamma'$ actually has no infinite rays attached, and $Q=\Gamma$. We will need only two properties of category $\wdtrep(\Gamma)$ (see~\cite[Th. II.1.3]{RvdB}): 
\begin{enumerate}
    \item[(*)] $\wdtrep(\Gamma)$ contains $\rep(\Gamma)$, the category of finitely presented right representations of~$\Gamma$, as an exact full subcategory, and 
    \item[(**)] $\wdtrep(\Gamma)=\rep(\Gamma)$ if $D^b(\rep(\Gamma))$ has a Serre functor. 
\end{enumerate}
Assume that $\Gamma$ contains an infinite ray~\eqref{eq_ray}. Then projective $\Gamma$-modules $P_0\supset P_1\supset P_2\supset\ldots$ give an infinite descending chain in $\rep(\Gamma)$ and by (*) in $\wdtrep(\Gamma)\cong \AA$. Since $\AA$ is a finite length category, we get a contradiction. Therefore, in $\Gamma=\Gamma'$ every vertex is contained only in finitely many paths, and $\rep(\Gamma)=\moddo \k \Gamma$. By the above arguments, $D^b(\rep(\Gamma))$ has a Serre functor. Using (**), we have 
$$\moddo\k Q=\AA \cong \wdtrep(\Gamma)=\rep(\Gamma) =\moddo \k \Gamma $$
By Remark~\ref{rem_uniquemoddo} we deduce $Q=\Gamma$.
\end{proof}

\begin{remark}
    It is interesting to note that for an acyclic quiver $Q$ with finite $Q_0$ and $Q_1$ one always has $D^b_0(\k Q)\cong D^b(\modd \k Q)$ and the Rouquier dimension (\cite{Rouquier}) of $D^b_0(\k Q)$ is $\le 1$. However, it is not clear what the dimension of $D^b_0(\k Q)$ is for acyclic quivers $Q$ with finite $Q_0$ but possibly infinite $Q_1$.
\end{remark}

\medskip
To recognize a quiver-like category, one needs to find objects, corresponding to simple modules. 
We leave to the reader the following easy, but very important computation:
\begin{lemma}
\label{lemma_comput}
Let $\AA=\Modd\k Q$ and let $s_i\in \AA$ 
be simple modules. Then
\begin{align}
    &\Hom_\AA(s_i,s_i)=\k, \quad \Hom_\AA(s_i,s_j)=0\quad \text{for}\quad i\ne j,\\
    &\dim\Ext^1_\AA(s_i,s_j)=\text{number of arrows in $Q$ from $j$ to $i$},\\
    &\Ext^p_\AA(s_i,s_j)=0\quad\text{for $p\ge 2$.}
\end{align}
\end{lemma}
It is now convenient to make
\begin{definition}
Let $\TT$ be a triangulated $\k$-linear category. A collection of
objects $t_i, i\in I$, in $\TT$ is called \emph{vertex-like} if 
$$\Hom(t_i,t_i ) = \k, \quad \Hom(t_i,t_j)=0\quad\text{for $i\ne j$,}\quad
\Hom^p(t_i,t_j) = 0\quad \text{ for all $i, j$ and
$p \ne  0, 1.$}$$ 
\end{definition}

A model example of a vertex-like collection is given (see Lemma~\ref{lemma_comput}) by simple modules~$s_i$ in  $D^b_0(\k Q)$, where $i$ runs through $Q_0$.

We are going to prove next that  any vertex-like collection generates a quiver-like triangulated category.  Lemma~\ref{lemma_comput} suggests how to define a quiver.
\begin{definition}
\label{def_extquiver}
    Let $\{t_i\}_{i\in I}$ be a vertex-like set in a triangulated category. Its \emph{$\Ext$-quiver} is defined as follows: take $Q_0=I$ and take  $\dim \Hom^1(t_j,t_i)$ arrows from $i$ to $j$ (this can be infinite, then $\dim$ denotes the cardinality of a basis). 
\end{definition}

The key argument is contained in the following 

\begin{prop}
\label{prop_39}
Let $\TT$ be an algebraic idempotent-complete $\k$-linear category. Assume~$\TT$ is generated by a vertex-like set $t_i, i\in I$, of objects. Let $\EE$ be the  dg category with  $\mathrm{Ob}(\EE)=I$, and with
$\Hom_\EE^p(i,j)= \Hom^p_\TT(t_i,t_j)$ for all $i,j\in I, p\in \Z$, viewed as a complex with zero differential.
Then one has an exact equivalence 
$\TT\to \Perf(\EE)$, sending~$t_i$ to the representable $\EE$-module  $h^{i}$.
\end{prop}
\begin{proof}
By \cite[Th. 7.6.0.6]{LH}, there exists a strictly unital minimal $A_\infty$-category $\EE$ with 
$\mathrm{Ob}(\EE)=I$ and an equivalence $\TT\to \Perf^\infty(\EE)$ of triangulated categories, sending $t_i$ to the representable $\EE$-module  $h^{i}$.
Moreover, by construction one has 
$\Hom_\EE^p(i,j)= \Hom^p_\TT(t_i,t_j)$ for all $i,j\in I, p\in \Z$. 

We claim that all  operations $m_n, n\ge 3$ vanish. This is  because of the grading: by the definition of a vertex-like collection, there are only scalar endomorphisms and $\Hom$-s of degree $1$ in $\EE$. More precisely, let $a_1,\ldots,a_n$ be homogeneous morphisms in $\EE$, $n\ge 3$. If some  $a_k$ has degree $0$, then $a_k\in \k\cdot 1_i$ for some object $i$ in $\EE$. Hence  
$m_n(a_1\otimes\ldots\otimes a_n)=0$ by strict unitality of $\EE$. Otherwise all $a_i$-s have degree $1$. 
Then 
$$\deg m_n(a_1\otimes\ldots\otimes a_n)=\deg m_n+\sum_{k=1}^n \deg a_k =2-n+n=2,$$
but there are no non-zero morphisms in degree $2$ in $\EE$. 

Hence $\EE$
is actually a dg category, and its differential $m_1$ is zero by the minimality assumption. It remains to note that the perfect derived category $\Perf^\infty(\EE)$ of $A_\infty$-modules over $\EE$ is equivalent to the perfect derived category $\Perf(\EE)$ of dg modules over $\EE$, see \cite[Cor. 4.1.3.11]{LH}.
\end{proof}

\begin{theorem}
\label{th_39}
Let $\TT$ be an algebraic $\k$-linear category. Assume $\TT$ is generated by a vertex-like set  $t_i, i\in I$, of objects. Let $Q$ be the $\Ext$-quiver of this set. Then $\TT$ is quiver-like: there exists an exact equivalence $\TT\to D^b_0(\k Q)$, sending $t_i$ to $s_i$.  
\end{theorem}
\begin{proof}
By Lemma~\ref{lemma_comput} 
\begin{equation}
\label{eq_TD}
    \Hom^p_{\TT}(t_i,t_j)\cong \Hom^p_{D^b_0(\k Q)}(s_i,s_j)\quad    \text{for all $i,j\in I$ and $p\in\Z$.}
\end{equation}
Assume first that $\TT$ is idempotent-complete. Then we can apply Proposition~\ref{prop_39} to categories $\TT$ and $D^b_0(\k Q)$ with vertex-like collections $\{t_i\}$ and $\{s_i\}$. By~\eqref{eq_TD}, both $\TT$ and $D^b_0(\k Q)$ are equivalent to the same category $\Perf(\EE)$, and for any $i\in I$ the objects $t_i\in \TT$ and $s_i\in D^b_0(\k Q)$ are sent to the same object in $\Perf(\EE)$. Hence the statement follows.

If $\TT$ is not known to be idempotent-complete, the same arguments prove that there is a fully faithful exact functor $\Phi\colon \TT\to D^b_0(\k Q)$, sending $t_i$ to $s_i, i\in I$. But objects $s_i$ generate $D^b_0(\k Q)$ as a triangulated category (without adding direct summands). Therefore the essential image of $\Phi$ is all $D^b_0(\k Q)$, and $\Phi$ is an equivalence.
\end{proof}
\begin{remark}
Theorem~\ref{th_39} was first proved in \cite[Proposition 3.9]{EL} under extra conditions:~$I$ was assumed finite and the dg endomorphism algebra of $\oplus t_i$ was assumed formal. 
\end{remark}    

We can say even more if a vertex-like set in an abelian category is given.

\begin{corollary}
\label{cor_40}
Let $\AA$ be a $\k$-linear abelian category. 
Let $\TT \subset  D^b(\AA)$ be a thick subcategory and $\SS=\TT\cap \AA$. Assume there is a vertex-like set $t_i, i\in I$, in $\SS$ that generates $\TT$ as a thick triangulated subcategory. Let $Q$ be the associated $\Ext$-quiver. Then 
\begin{enumerate}
\item There are equivalences
$$D^b_0(\k Q)\xra{\sim} \TT\quad\text{and}\quad \moddo \k Q\xra{\sim} \SS,$$
sending $s_i$ to $t_i$. 
\item Objects $t_i$ represent isomorphism classes of simple objects in $\SS$.
\item $\SS$ is the smallest subcategory in~$\AA$ that contains $t_i$-s and is closed under extensions.
\item The natural functor $D^b(\SS)\to D^b(\AA)$ is fully faithful with the essential image $\TT$. 
\end{enumerate}
\end{corollary}
\begin{proof}
    Note that $\TT$ is an algebraic triangulated category, so we can apply Theorem~\ref{th_39}. We get that there is an exact equivalence $\TT\to D^b_0(\k Q)$. Recall equivalence $D^b(\moddo \k Q)\to D^b_0(\k Q)$ from  Proposition~\ref{prop_moddo}. Let
    $$\Phi\colon D^b(\moddo \k Q)\to \TT$$
    be the resulting equivalence. Note that $\Phi$ sends simple modules   $s_i$ to objects $t_i$. Any module  in $\moddo \k Q$ is an iterated extension of $s_i$-s, hence $\Phi$ sends the subcategory   $\moddo \k Q\subset D^b(\moddo \k Q)$ to $\SS\subset \TT$.  Let us check that the restriction 
    $$\phi=\Phi|_{\moddo\k Q}\colon \moddo\k Q\to\SS$$ is essentially surjective. 
    Recall that any  object in $D^b(\moddo \k Q)$ is the direct sum of it cohomology. Let $F\in \SS$ and $E=\Phi^{-1}(F)$. Then
    $$F\cong\Phi(E)\cong \Phi(\oplus_i H^i(E)[-i])\cong \oplus_i \phi(H^i(E))[-i].$$
    It follows that $\phi(H^i(E))=0$ and $H^i(E)=0$ for $i\ne 0$. Hence $E\cong H^0(E)\in \moddo\k Q$.
    Consequently, $\Phi$ restricts to an equivalence $\phi\colon \moddo \k Q\to\SS$. 

    (2) and (3) follow easily from equivalence $ \moddo \k Q\to\SS$. 
    
    For (4) it suffices to check that the derived functor 
    $$\Phi':=D(\phi)\colon D^b(\moddo\k Q)\to D^b(\AA)$$
    is fully faithful with the image $\TT$. While we are unsure if $\Phi'$ is isomorphic to $\Phi$, we note that 
    $\Phi'$ and $\Phi$ coincide on objects $s_i[d]$, where $i\in I, d\in\Z$. It follows that $\Phi'$ is fully faithful on such objects along with $\Phi$, and by standard d{\'e}vissage technique $\Phi'$ is fully faithful on its domain. Finally, we see that the image of $\Phi'$ is generated by objects $t_i[d]$ and therefore is $\TT$.
\end{proof}

We find it convenient for applications  to formulate another consequence of Theorem~\ref{th_39}.
Note that Corollary~\ref{cor_39} appears in~\cite[Proposition 3.10]{EL} with extra assumptions including finiteness. 
\begin{corollary}
\label{cor_39}
Assume that the field $\k$ is algebraically closed. Let $\AA$ be a $\k$-linear abelian hereditary
 category. Let $\TT \subset  D^b(\AA)$ be a $\Hom$-finite thick essentially small subcategory, and $\SS=\TT\cap \AA$. Assume there exists a linear function
$$r \colon K_0(\SS) \to \Z,$$
such that for any non-zero $F \in \SS$ one has $r([F]) > 0$. Then the category $\SS$ is  of finite length, its simple objects form  a vertex-like collection generating $\TT$, and categories $\SS$ and $\TT$ are quiver-like.
\end{corollary}
\begin{proof}
By Proposition~\ref{prop_hersubcat1} $\SS$ is an abelian hereditary category. 
One can prove that any object $F$ in $\SS$ has finite length by induction in $r([F])$. Indeed, if $r([F])=0$ then $F=0$ and there is nothing to prove, and if $0\to F'\to F\to F''\to 0$ is a non-trivial exact sequence, then $r([F])=r([F'])+r([F''])$, $r([F']),r([F''])<r([F])$, and one can proceed by induction.

Let $t_i, i\in I$, denote the set of (isomorphism classes of) simple objects in $\SS$. Then  $t_i$-s generate $\SS$ (because  any object in $\SS$ is an iterated extension of simples) and $\TT$ (because any object in $\TT$ is a direct sum of its cohomology objects, which are in $\SS$). Further, $t_i$-s form a vertex-like family. Indeed, $\Hom^p_{\TT}(t_i,t_j)=\Hom^p_\AA(t_i,t_j)=0$ for $p\ne 0,1$ since  $\AA$ is hereditary, $\Hom(t_i,t_j)=0$ for $i\ne j$ because they are simple objects. Finally, $\End(t_i)$ is a finite-dimensional division $\k$-algebra (by Schur's lemma and our assumptions). Since $\k$ is algebraically closed,  $\End(t_i)$ is isomorphic to~$\k$. Now by  Corollary~\ref{cor_40}  $\SS$ and $\TT$ are quiver-like. 
\end{proof}

\medskip
Now we provide an analogue of Theorem~\ref{th_39}, giving a characterisation of abelian quiver-like categories.
\begin{theorem}
\label{th_39ab}
    An abelian $\k$-linear category $\AA$ is quiver-like if and only if $\AA$ is essentially small, has finite length,  $\End(S)= \k$ for any simple object $S\in\AA$, and  
    $\Ext^j(S_1,S_2)=0$ for $j\ge 2$ and for any  simple objects $S_1,S_2\in\AA$.
\end{theorem}
\begin{proof}
    ``Only if'' part is trivial, let's prove  ``if'' part. Consider $D^b(\AA)$, it is an algebraic triangulated category. Choose a representative set $t_i, i\in I,$ for isomorphism classes of simple objects in $\AA$.  Then $\{t_i\}_{i\in I}$ is a vertex-like set by our assumptions. Note that this set generates $D^b(\AA)$ as a triangulated category: any object in $\AA$ has a finite filtration with quotients isomorphic to some $t_i$, hence belongs to the triangulated subcategory in $D^b(\AA)$ generated by $t_i$-s.  
    So we can apply Corollary~\ref{cor_40} with $\TT=D^b(\AA)$ and deduce that $\AA$ is quiver-like.
\end{proof}

\begin{remark}
Note that we do not require $\AA$ to be hereditary in Corollary~\ref{cor_40} and Theorem~\ref{th_39ab}.
\end{remark}

\section{Linear quivers, tubes, and their subcategories}
\label{section_linestubes}
Here we collect necessary facts about two families of abelian hereditary categories: modules over linear quivers and tubes, needed for the sequel.
We refer to \cite[1.7, 1.8]{ChenKrause}, \cite[Section 4]{Cheng}, \cite{Krause_strings}, or \cite{Di} for details. 

\subsection{Linear quivers}
For $n\ge 1$, let $\AA_n$ be the category of right representation of the quiver of type $A_n$:
$$\xymatrix{\bul_1\ar[rr] && \bul_2\ar[rr] && \bul\ar[r] & \ldots \ar[r]& \bul\ar[rr] && \bul_n} $$
For $n=0$ we put $\AA_0=0$. Categories $\AA_n$ are connected, hereditary, and uniserial. They have finitely many indecomposable objects, all off which are exceptional. Despite of their omnipresence,  categories $\AA_n$ do not seem to have a standard name. Thick subcategories in $\AA_n$ and $D^b(\AA_n)$ are  well-understood, we summarize relevant facts in 
\begin{prop}[See~{\cite{Di}, \cite{Krause_strings}}]
\label{prop_thicklines}
    Let $n\ge 1$.
    \begin{enumerate}
    \item Let $E\in \AA_n$ be an indecomposable object of length $m, 1\le m\le n$. Then $E^\perp_{\Ab}$  is equivalent to $\AA_{n-m}\times \AA_{m-1}$, and
        $D^b(\AA_n)=\langle D^b(E^\perp_{\Ab}), \langle E\rangle\rangle$.
    \end{enumerate}
    Let $\TT\subset D^b(\AA_n)$ be a thick subcategory, and  $\SS=\TT\cap \AA_n$. Then
    \begin{enumerate}
    \item[(2)] $\TT$ is generated by an exceptional collection in $\SS$, in particular, $\TT$ is admissible.
    \item[(3)] One has $$\SS\cong \AA_{n_1}\times\ldots\times \AA_{n_k},$$ where $k\ge 0$ and $n_i\ge 1$.
    \end{enumerate}
\end{prop}
\begin{proof}
    We provide a short proof to illustrate methods from Section~\ref{section_QL}. 
    
    For $1\le i\le j\le n$, denote by $M_{ij}$ the indecomposable right representation of $A_n$, concentrated in vertices $i,i+1,\ldots,j-1,j$. We can assume $E=M_{pq}$ with $m=q-p+1$. 
    One can check that $E^\perp$
    is generated by the vertex-like family of modules
    \begin{equation}
    \label{eq_XXX}
    M_{11},M_{22},\ldots,M_{p-2,p-2}, M_{p-1,q},M_{q+1,q+1},\ldots,M_{nn}\: ; M_{pp},M_{p+1,p+1},M_{q-1,q-1}    
    \end{equation}
    Note that the $\Ext$-quiver of~\eqref{eq_XXX} is $A_{n-m}\sqcup A_{m-1}$. By Corollary~\ref{cor_40}, we get an equivalence 
    $$E^\perp_{\Ab}\cong \moddo \k (A_{n-m}\sqcup A_{m-1})\cong \AA_{n-m}\times \AA_{m-1}.$$
    By Proposition~\ref{prop_hersubcat2} $E^\perp\cong D^b(E^\perp_{\Ab})$, this proves (1).

    (2) is clear since any thick subcategory in $\AA_n$ contains an indecomposbale object, which is necessarily   exceptional. (3) follows from (1) by induction, because $\TT$ is the orthogonal to an exceptional collection.
\end{proof}

\subsection{Tubes}
\label{section_tubes}
Let $n\ge 1$, let $Z_n$ be the oriented cycle with $n$ vertices: 
\begin{equation*}
\xymatrix{\bul_1\ar[rr] && \bul_2\ar[rr] && \bul\ar[r] & \ldots \ar[r]& \bul\ar[rr] && \bul_n\ar@(dl,dr)[llllllll] \\
&&&&&&&&}
\end{equation*}
Let
$\UU_n:=\moddo \k Z_n$.  It is a connected abelian uniserial hereditary category. Moreover, $\UU_n$ has Serre duality in the form 
\begin{equation}
\label{eq_serreU}
\Ext^1(X,Y)^*\cong \Hom(Y,\tau X).
\end{equation}
where  $X,Y\in \UU_n$, and $\tau\colon \UU_n\to \UU_n$ is the autoequivalence induced by the rotation  of $Z_n$ against the arrows.  Note that $\tau^n=\id$. The derived category $D^b(\UU_n)$ has a Serre functor, given by $\tau[1]$.
Let $S$ be some simple object in $\UU_n$, then all simple objects are $S,\tau S,\ldots, \tau^{n-1}S$. Further, for a simple object $S$ and $i\ge 1$, denote by $S^{[i]}$ the unique  indecomposable object $M$ with $M/\mathrm{rad}(M)\cong S$ and of length $i$. Such  $S^{[i]}$ fits into a non-trivial extension 
$$0\to \tau S^{[i-1]}\to S^{[i]}\to S\to 0.$$
Then all indecomposable objects of length $i$ in $\UU_n$ are $S^{[i]}, \tau S^{[i]}, \ldots, \tau^{n-1}S^{[i]}$. In particular, there are infinitely many indecomposable objects in $\UU_n$.
We will call $\UU_n$ a \emph{tube of rank $n$}.

Object $S^{[i]}$ is exceptional if and only if $i<n$, and is sphere-like  if and only if $i=n$. 
For $i>n$ the object $S^{[i]}$ is not exceptional nor sphere-like, it has nilpotent endomorphisms.

One can show that any connected abelian category of finite length with Serre duality \eqref{eq_serreU} and with finitely many simples is equivalent to $\UU_n$ for some $n\ge 1$, see~\cite{RvdB}.

\medskip
Thick subcategories in tubes have  also been studied, see~\cite{Krause_strings}, \cite{Cheng}, or \cite{Di}.
Here we collect the results from these sources that we will use later.

\begin{prop}
\label{prop_thicktubes}
    Let $n\ge 1$. 
    \begin{enumerate}
        \item Let $E\in \UU_n$ be an indecomposable object of length $m, 1\le m\le n$. Then $E^\perp_{\mathrm Ab}$ is equivalent to $\UU_{n-m}\times \AA_{m-1}$, 
         and we have  the semi-orthogonal decomposition
        $D^b(\UU_n)=\langle D^b(E^\perp_{\Ab}), \langle E\rangle\rangle$ (where $\UU_{n-m}$ does not appear if $m=n$).
    \end{enumerate}
    Let $\TT\subset D^b(\UU_n)$ be a thick subcategory, and  $\SS=\TT\cap \UU_n$. Then 
    \begin{enumerate}
        \item[(2)] $\TT$ is admissible and exactly one of the following holds:
        \begin{enumerate}
            \item If  $\TT$ contains no sphere-like objects then $\TT$ is generated by an exceptional collection in $\SS$. 
            \item If $\TT$ contains a sphere-like object then   $\TT$ is generated by a semi-orthogonal collection $E_1,\ldots,E_m$ in $\SS$, where $E_1,\ldots,E_{m-1}$ are exceptional and $E_m$ is sphere-like. 
        \end{enumerate}
         
        \item[(3)] According to (2), one has:
        \begin{enumerate}
            \item 
            \begin{equation}
            \label{eq_AAAA}
                \SS\cong \AA_{n_1}\times\ldots\times \AA_{n_k},
            \end{equation} where $k\ge 0$, $n_i\ge 1$. All indecomposable objects in $\TT$ are exceptional. 
            \item 
            \begin{equation}
                \label{eq_UAAA}
                \SS\cong \UU_s\times \AA_{n_1}\times\ldots\times \AA_{n_k},
            \end{equation}where   $s\ge 1$, $k\ge 0$, $n_i\ge 1$.
             $\TT$ is not generated by an exceptional collection.  
        \end{enumerate}
        Moreover, $\TT$ is of type (a)  if and only if $\TT^\perp$ (and $^\perp\TT$) is of type (b).
    \end{enumerate}
\end{prop}
\begin{proof}
    As in Proposition~\ref{prop_thicklines} we sketch a proof, cf.~\cite[Sect. 4]{Cheng}.

    We can assume $E=S^{[m]}$ for a simple object $S$. 
    One can check that $E^\perp$
    is generated by the vertex-like family of modules
    \begin{equation}
    \label{eq_SSS}
    \tau^{-1}S, \tau^{-2}S, \ldots, \tau^{m+1}S, S^{[m+1]}\: ; \tau S, \tau^2S,\ldots,\tau^{m-1}S
    \end{equation}
    Note that the $\Ext$-quiver of~\eqref{eq_SSS} is $Z_{n-m}\sqcup A_{m-1}$. By Corollary~\ref{cor_40}, we get an equivalence 
    $$E^\perp_{\Ab}\cong \moddo \k (Z_{n-m}\sqcup A_{m-1})\cong \UU_{n-m}\times \AA_{m-1}.$$
    If $n=m$ then the first group in~\eqref{eq_SSS} contains $n-m$ modules, so is empty, and  $Z_{n-m}$ above does not occur. Note that $E$ together with \eqref{eq_SSS} generate $D^b(\UU_n)$, and  we  deduce~(1). In particular,  $\langle E\rangle$ is right, and similarly left admissible.

    Any thick subcategory $\SS$ in $\UU_n$ contains a simple in $\SS$ object $E$, which must be exceptional or sphere-like. Passing to $E^\perp$ and arguing by induction using (1), we see that any thick subcategory $\TT\subset D^b(\UU_n)$ is generated by a semi-orthogonal sequence of exceptional and sphere-like objects, hence is admissible.  Moreover, assume $\TT$ contains a sphere-like object $E$, then $E^\perp\cong D^b(\AA_{n-1})$ by (1). Hence $\TT\cap E^\perp$ has no sphere-like objects, and one can find a generating semi-orthogonal sequence for $\TT$ with only one sphere-like object as required. This explains (2).

    For (3), use that $\TT$ is admissible and thus $\TT=(^\perp\TT)^\perp$. Consider two cases: $\TT$ is the right orthogonal to a subcategory as in (2a) or as in (2b). 
    Use induction, (1), and Proposition~\ref{prop_thicklines} to see that in the first case $\SS$ is equivalent to~\eqref{eq_UAAA} and hence of type (b), and in the second case $\SS$ is equivalent to~\eqref{eq_AAAA} and hence of type (a).   
\end{proof}

Nice descriptions of thick subcategories in linear quivers and tubes in terms of non-crossing arc arrangements and non-crossing partitions have been found by Ingalls -- Thomas~\cite{IT} for $\AA_n$ and Dichev~\cite{Di} and Krause~\cite{Krause_strings} for $\UU_n$ respectively.

\section{Weighted projective curves}
\label{section_wpc}

Here we collect necessary facts about weighted projective curves and associated categories of coherent sheaves.
In this section and further we assume that the base field $\k$ is  algebraically closed.

\subsection{Weighted projective curves}
\label{section_wpcdef}
Let $X$ be a smooth projective curve over $\k$ and $w\colon X\to \N$ be a function on the set of closed points of $X$ such that $w=1$ except for finitely many points $x_1,\ldots,x_n$.  Function $w$ is called the \emph{weight function}, points $x\in X$ with $w(x)=1$ are called \emph{regular} or \emph{homogeneous}, points $x_1,\ldots,x_n$ are called \emph{orbifold} (or \emph{non-homogeneous, stacky, weighted, singular}, etc), and numbers $r_i:=w(x_i)\ge 2$ are called \emph{multiplicities} or \emph{weights}. The pair $\X=(X,w)$ will be called a \emph{weighted projective curve}.

\subsection{Category of coherent sheaves}
\label{section_wpccoh}
To a weighted projective curve $\X$ one can associate an abelian category $\coh\X$, called \emph{the category of coherent sheaves on $\X$}. It can be defined in several ways:
\begin{itemize}
\item One can construct a Deligne -- Mumford stack $\mathcal X$ from $(X,w)$ by applying \emph{root construction} with centres $x_i$ and multiplicities $r_i$, see e.g. \cite{Cadman}. Then define $\coh \X$ as the category $\coh \mathcal X$ of coherent sheaves on the stack $\mathcal X$.
\item One can obtain $\coh\X$  from $\coh X$ be \emph{inserting weights} $r_1,\ldots,r_n$ at points $x_1,\ldots,x_n$ in purely categorical terms. Let us describe briefly one step of this procedure. Let $\X=(X,w)$ be a  weighted curve, $x\in X$ be  a point, and $p\ge 1$. Denote  $\X'=(X,w')$ with $w'(x)=pw(x)$ and $w'(y)=w(y)$ for $y\ne x$. The idea of obtaining  $\coh\X'$ from  $\coh\X$ is to formally take $p$-th root of the twist functor $\s_x\colon \coh\X\to\coh\X$ and of the functor morphism $\eta_x\colon \id\to \s_x$, see Definition~\ref{def_sx} and Proposition~\ref{prop_sx}.
We refer to~\cite[Section 4]{Lenzing_fda_sing} for details.
\item If $n\ge 3$, there exists a smooth projective curve $Y$,  a finite group $G$ and a faithful $G$-action on $Y$ such that $Y/G\cong X$ and for any $y\in Y$ the stabilizer of $y$ is a cyclic group of order $w(\pi(y))$ (where $\pi\colon Y\to X$ is the quotient morphism). Then the stack $\mathcal X$ is isomorphic to the quotient stack $Y/\!\!/G$ and one can define $\coh\X$ as the category of $G$-equivariant coherent sheaves on $Y$. See \cite{Lenzing} for details.
\item Following an approach by Artin and Zhang~\cite{AZ}, one can consider one-dimensional non-commutative projective schemes. Let $R=\oplus_{i\ge 0} R_i$ be a finitely generated commutative associative unital graded $\k$-algebra with $R_0=\k$. Assume $R$ is an isolated singularity, has Krull dimension $2$ and is graded Gorenstein.   Then the quotient category 
$$\mathrm{qgr}(R):=\frac{\mathrm{grmod{-}}R}{\mathrm{grmod_0{-}}R}$$
of finitely generated graded $R$-modules by the subcategory of finite-dimensional modules is equivalent to the category of coherent sheaves on weighted projective curve $(\Proj(R),w)$ with certain $w$. Moreover, any weighted projective curve can be obtained in this way.
\end{itemize}

Unfortunately, none of the above definitions is straightforward. Alternatively, following Lenzing, one can define coherent sheaves on weighted projective curves axiomatically, as abelian categories satisfying some properties, see~\cite{LenzingReiten}, \cite{LenzingFuchs}, \cite{RvdB}. 

\begin{definition}
\label{def_HH}
A \emph{category of coherent sheaves on a weighted projective curve} is a small $\k$-linear category $\HH$ satisfying the following axioms:
    \begin{enumerate}
        \item[(H1)] $\HH$ is abelian, connected, and any object in $\HH$ is Noetherian,
        \item[(H2)] $\HH$ is $\Ext$-finite: all $\Ext$ spaces are finite-dimensional,
        \item[(H3)] $\HH$ satisfies Serre  duality in the following sense: there exists an autoequivalence $\tau\colon \HH\to \HH$  and natural  isomorphisms 
        \begin{equation}
        \label{eq_SerreDual}
            \Ext^1(F_1,F_2)\cong \Hom(F_2,\tau F_1)^*
        \end{equation}
        of vector spaces for all $F_1,F_2\in\HH$,
        \item[(H4)] $\HH$ contains an object of infinite length.
    \end{enumerate}
Denote by $\HH_0\subset \HH$ the full subcategory of objects having finite length. Clearly $\HH_0$ is preserved by $\tau$. We additionally require
     \begin{enumerate}
        \item[(H5)] Any object in $\HH/\HH_0$ has finite length, and
        \item[(H6)] There are infinitely many $\tau$-orbits of simple objects in $\HH$.
    \end{enumerate}
\end{definition}
\begin{remark}
\begin{enumerate}
    \item It follows from condition (H3) that the derived category $D^b(\HH)$ has  a Serre functor, which   is given by $\tau[1]$ (the composition of $\tau$ and the cohomological shift by $1$). Moreover, (H3) is equivalent to the following: $\HH$ is hereditary,  $D^b(\HH)$ has a Serre functor, and  there are no non-zero projective objects in $\HH$. 
    \item It can be shown that (H1)-(H4) actually imply (H5).
    \item Condition (H6) is needed to exclude some degenerate examples.
\end{enumerate}    
\end{remark}
In fact, any category, satisfying (H1)-(H6), is equivalent to $\coh\X$ for some weighted projective curve $\X=(X,w)$, see \cite{RvdB}.
Below  we explain how to associate a curve with a weight function to a category from Definition~\ref{def_HH}.

\subsection{First properties of coherent sheaves on weighted projective curves: points, torsion and torsion-free sheaves}

From now on, we assume that $\HH$ is a category satisfying conditions (H1)-(H6). We refer to~\cite{LenzingFuchs}, \cite{LenzingReiten}, \cite{RvdB} for the properties of $\HH$  collected below.
Objects in $\HH$ (resp. $\HH_0$) will be called \emph{sheaves} (resp. \emph{torsion sheaves}). A sheaf  is called \emph{torsion-free}, or a \emph{vector bundle}, if it has no torsion  subsheaves. The full subcategory of torsion-free sheaves will be denoted $\HH_+\subset \HH$. Any sheaf is a direct sum of a torsion sheaf and a torsion-free sheaf. 
 
Category $\HH_0$ decomposes into a direct sum
$$\HH_0=\oplus_{x\in X} \HH_x$$
of connected thick (and moreover, Serre) subcategories, where index set $X$ is the set of closed points of a smooth projective curve, also denoted by $X$.

Autoequivalence $\tau\colon \HH\to \HH$ is called \emph{translation}, it restricts to autoequivalences on $\HH_0$ and $\HH_+$. 
Any subcategory $\HH_x$ is $\tau$-stable, and is equivalent to the tube $\UU_{w(x)}$ for some $w(x)\ge 1$. 
Moreover, for all points $x\in X$ except for a finite number one has $w(x)=1$. Thus $w$ is indeed a weight function, and we recover a weighted projective curve $(X,w)$ as defined in Section~\ref{section_wpcdef}. From Section~\ref{section_tubes} we know that $\HH_x$ has $w(x)$ simple sheaves, and they belong to the same $\tau$-orbit. In other words, $\tau$-orbits of simple objects of $\HH$ correspond to points of $X$.

Section~\ref{section_tubes} implies 
\begin{lemma}
    \label{lemma_exceptionalatorbifold}
    Let $E\in \HH_x$ be an indecomposable torsion sheaf of length $i$.
    \begin{enumerate}
    \item $E$ is exceptional if and only if $i<w(x)$. In particular, $\HH_x$ contains exceptional   sheaves if and only if $x$ is an orbifold point.
    \item $E$ is sphere-like if and only if $i=w(x)$. Moreover,  any category $\HH_x$ contains $w(x)\ge 1$ sphere-like sheaves.
    \end{enumerate}
\end{lemma}

For a non-zero sheaf $F\in \HH_x$ we say that $F$ is \emph{supported} at $x$. For a general torsion sheaf~$F$ its support $\Supp(F)\subset X$ is defined as the union of supports of the indecomposable summands of $F$. For a sheaf $F$ that is not torsion we say that $\Supp(F)=X$.

The quotient category $\HH/\HH_0$ is equivalent to the category of finite-dimensional vector spaces over the field
$\k(X)$ of rational functions on the curve $X$. The \emph{rank} of a vector bundle $V\in \HH$ is defined as the dimension of its image in $\HH/\HH_0$. Vector bundles of rank one are called \emph{line bundles}. Any vector bundle has a filtration by line bundles. For convenient reference we state

\begin{lemma}
\label{lemma_homstorsion}
\begin{enumerate}
\item For any $V\in \HH_+, F\in \HH_0$  one has
$$\Hom(F,V)=\Ext^1(V,F)=0.$$
\item For any non-zero $V\in \HH_+$ and any $x\in X$ there is some simple $S\in \HH_x$ such that 
$\Hom(V,S)\ne 0$. Moreover, 
$$\sum_{S\in \Gamma_0(\HH_x)} \dim\Hom(V,S)=\rank (V).$$
\item For any non-zero $V\in \HH_+$, any $x\in X$, and any sphere-like $M\in \HH_x$ one has
$\Hom(V,M)\ne 0$. Moreover, 
$$ \dim\Hom(V,M)=\rank (V).$$
\end{enumerate}
\end{lemma}

\subsection{Twist functors}
\label{section_twist}

Here we recall definitions and properties of two families of autoequivalences $\s_x$ and $c_x$ of $\coh \X$, where $x\in X$. Functors $\s_x$ are often called mutations. We consider their iterations  $c_x=\s_x^{w(x)}$ and call them  twists, they should be viewed as the tensor multiplication by the line bundles $\O_{\X}(x)$. We will need twists in Proposition~\ref{prop_bigbig}.

\begin{definition}[See  {\cite[Th. 10.8]{Lenzing_hercat}}, also {\cite{Lenzing_fda_sing}}, {\cite{Meltzer}}]
\label{def_sx}
Let  $x\in X$ be a point. One can define an autoequivalence
$\s_x\colon D^b(\HH)\to D^b(\HH)$ given for an object $F$ by the exact triangle
\begin{equation}
    \label{eq_sigma}
    \oplus_{S\in \Gamma_0(\HH_x)}\Hom^{\bul}(S,F)\otimes S\to F\to \s_x(F)\to \oplus_{S\in \Gamma_0(\HH_x)}\Hom^{\bul}(S,F)\otimes S[1],
\end{equation}
where $S$ runs over the set $\Gamma_0(\HH_x)$ of isomorphism classes of simple sheaves supported at~$x$.
The inverse functor is given by the exact triangle
\begin{equation}
    \label{eq_sigmaminus}
    \s_x^{-1}(F)\to F\to \oplus_{S\in \Gamma_0(\HH_x)}\Hom^{\bul}(F,S)^*\otimes S \to \s_x^{-1}(F)[1], 
\end{equation}
where $\Hom^{\bul}(F,S)^*$ is the dual graded $\k$-vector space: $(\Hom^{\bul}(F,S)^*)^i=\Hom^{-i}(F,S)^*$.
\end{definition}

Conceptual explanation for the fact that these formulas define mutually inverse functors is that $\s_x$ is a spherical twist in the sense of \cite{Logvinenko}. Indeed, let 
$$\AA=\add\left(\bigoplus_{S\in \Gamma_0(\HH_x)}S\right)\subset \HH$$ 
be the additive envelope of simple objects in $\HH_x$, this is a semi-simple abelian category. Then the functor $D^b(\AA)\to D^b(\HH)$ induced by the inclusion $\AA\to \HH$ is spherical, and the associated spherical twist is $\s_x$, see \cite[Lemma 2.7]{KuzPer}.

\begin{prop}[See \cite{Meltzer}, \cite{Lenzing_fda_sing}]
\label{prop_sx}
Functors $\s_x$ have the following properties
\begin{enumerate}
    \item $\s_x$ restricts to an autoequivalence $\HH\to \HH$,
    \item $\s_x$ preserves subcategories $\HH_0, \HH_y, \HH_+\subset \HH$ for all $y\in X$,
    \item $\s_x=\tau^{-1}$ on $\HH_x$, $\s_x=\id$ on $\HH_y$ for $y\ne x$,
    \item for $V$ a vector bundle, \eqref{eq_sigma} is the $\HH_x$-universal extension
    $$0\to V\to \s_x(V)\to \oplus_{S\in \Gamma_0(\HH_x)}\Ext^1(S,V)\otimes S\to 0,$$
    \item for $V$ a vector bundle, \eqref{eq_sigmaminus} is the $\HH_x$-universal semi-simple quotient
    $$0\to \s_x^{-1}(V)\to V\to \oplus_{S\in \Gamma_0(\HH_x)}\Hom(V,S)^*\otimes S\to 0,$$
    \item \eqref{eq_sigma} defines a natural transformation of functors $\eta_x\colon \id\to \s_x$.
\end{enumerate}
\end{prop}

We study the universal semi-simple quotients in more details.
For any $x\in X$ and any $F\in \HH$ we define $\Top_x(F)$ as the maximal semi-simple quotient of $F$, supported at $x$. Explicitly, $\Top_x(F)$ is given by the natural morphism 
\begin{equation}
    \label{eq_top}
    F\to \oplus_{S\in \Gamma_0(\HH_x)} S\otimes \Hom(F,S)^*=\Top_x(F).
\end{equation}
Recall that for a vector bundle $V$ one has exact triple
\begin{equation}
    \label{eq_sigmatop}
    0\to \s^{-1}_x(V)\xra{\eta_x} V\to \Top_x(V)\to 0,
\end{equation}
and
$$\len(\Top_x(V))=\sum_{S\in \Gamma_0(\HH_x)} \dim\Hom(V,S)=\rank(V)$$
by Lemma~\ref{lemma_homstorsion}.
We arrive at  useful
\begin{lemma}
    \label{lemma_toprank}
    Let $V$ be a vector bundle, $x\in X$ and $Q\in \HH_x$ be a quotient sheaf  of $V$. Then $Q$ has at most $\rank(V)$ indecomposable summands. 
\end{lemma}
\begin{proof}
    Let $Q=\oplus_{i=1}^n Q_i$ with non-zero  $Q_i$. Then any $\Top_x(Q_i)$ is non-zero and semi-simple and there is a surjection $V\to \Top_x(Q)=\oplus_{i=1}^n \Top_x(Q_i)$. By the above discussion, $\Top_x(Q)$ is the quotient of $\Top_x(V)$ and $n\le \len(\Top_x(Q))\le \len(\Top_x(V))=\rank(V)$.
\end{proof}

\begin{definition}
For a vector bundle $V$ and $m\ge 1$, consider the composition of injections
$$\eta_x^m\colon \s^{-m}(V)\xra{\eta_x} \s^{-m+1}(V)\xra{\eta_x} \ldots \s^{-1}(V)\xra{\eta_x} V.$$
Denote by $\Top^{[m]}_x(V)$ its cokernel.  
\end{definition}

\begin{lemma}
    \label{lemma_topm}
    Let $V$ be a vector bundle, $x\in X$ and $Q=\Top_x^{[m]}(V)$. Then $Q\cong \oplus_{i=1}^r Q_i$
    where $r=\rank(V)$, any $Q_i$ is an indecomposable torsion sheaf supported at $x$ and of length $m$. Moreover, assume $$\Top_x(V)\cong \oplus_{i=1}^r S_i$$ 
    with $S_i$ simple, then (recall notation from Section~\ref{section_tubes})
    $$\Top^{[m]}_x(V)\cong \oplus_{i=1}^r S_i^{[m]}.$$
\end{lemma}
\begin{proof}
Let $Q\cong \oplus_{i=1}^r Q_i$ with indecomposable $Q_i$. By Lemma~\ref{lemma_toprank}, $r\le \rank(V)$. By construction, $Q=\Top^{[m]}_x(V)$ has a filtration $F_j$ of length $m$ with semi-simple factors $\Top_x(\s^{-j}_x(V))$, $j=0,\ldots,m-1$.
Intersecting $F_j$ with  $Q_i$  one gets a filtration of $Q_i$ of length $m$ with semi-simple factors. Since $Q_i$ is uniserial, these factors are simple (if non-zero), and $\len(Q_i)\le m$.   Also we see that the length of $Q$ is $m\cdot \rank (V)$. That is,
$$m\cdot \rank(V)=\len(Q)=\sum_{i=1}^r\len(Q_i)\le mr \le m\cdot \rank(V).$$
It follows that we have $\len(Q_i)=m$ for all $i$ and $r=\rank(V)$. We have 
$$\oplus_i S_i\cong \Top_x(V)\cong \Top_x(Q)\cong \oplus_i \Top_x(Q_i),$$
and we may assume, up to renumbering, that $\Top_x(Q_i)\cong S_i$. Since $\UU_x$ is a uniserial category, it follows that $Q_i\cong S_i^{[m]}$.
\end{proof}

\begin{definition}
\label{def_cx}
    For $x\in X$, consider the autoequivalence $c_x\colon \HH\to \HH$ given by
    $$c_x=\s_x^{w(x)}.$$
    We will call $c_x$ a \emph{twist}. 
\end{definition}
Note that one has an exact sequence
\begin{equation}
\label{eq_ctop}
    0\to c_x^{-1}(V)\to V\to \Top_x^{[w(x)]}(V)\to 0
\end{equation}
for any vector bundle $V$.

\begin{lemma}
\label{lemma_cx}
    The twist~$c_x$ preserves subcategories  of torsion sheaves and torsion-free sheaves, and preserves rank and support of sheaves. Moreover,  $c_x$ acts by identity on torsion sheaves.
\end{lemma}
\begin{proof}
    Follows from the definition and Proposition~\ref{prop_sx}. For the last statement, let $F\in \HH_x$, then $c_x(F)=\s_x^{w(x)}(F)\cong \tau^{-w(x)}F\cong F$ since $\tau$ has order $w(x)$ on $\HH_x$. For $y\ne x$ and $F\in \HH_y$, one has $\s_x(F)\cong F$ and  $c_x(F)=\s_x^{w(x)}(F)\cong F$.
\end{proof}
\begin{remark}
    Geometrically, let $\HH$ be the category of coherent sheaves on a weighted projective curve $\X=(X,w)$, then $c_x$ is the tensor twist with the line bundle $\O_X(x)$. In particular, for $X\cong \P^1$, twists $c_x$ do not depend on the point $x$ up to an isomorphism. 
\end{remark}

\subsection{Orthogonal subcategories and reduction of weights}

Here we describe the orthogonal subcategory to an exceptional torsion sheaf, this is a crucial instrument for our further considerations. This description is well-known to experts, at least in the case of a weighted projective line (=rational curve), or for $m=1$,  since \cite{GL_perpen}. We find it convenient to sketch  a proof for completeness, which is analogous to \cite[Prop. 6.5]{Cheng}, where~$\X$ is supposed to be a weighted projective line. 
\begin{prop}
\label{prop_excort}
Let $\HH=\coh(\X)$ be a category of coherent sheaves on a weighted projective curve $\X=(X,w)$. Let $E$ be an exceptional torsion sheaf supported  at point $x\in X$, and  $m$ be its length. 
Let $E^\perp_{\Ab}=E^\perp\cap \HH\subset \HH$ be the orthogonal full subcategory. Then 
$$E^{\perp}_{\Ab}=\SS_1\times \SS_2, \quad\text{where}\quad \SS_1\cong \coh (X,w'), \SS_2\cong \AA_{m-1},$$
and $w'(x)=w(x)-m$, $w'(y)=w(x)$ for $y\ne x$.
Moreover, the induced functor $\coh(X,w')\to \coh(X,w)$ preserves the rank and the support of sheaves, and $\SS_2$ is supported at $x$.
\end{prop}
\begin{proof}
    We consider the case $m=1$ first. For $X\cong\P^1$ this is~\cite[Th. 9.5]{GL_perpen}. We give  a proof in general case, basing on axioms from Definition~\ref{def_HH}.

    Assume $E$ is a simple exceptional sheaf, supported at point $x$ of weight $w(x)=r\ge 2$.  Denote $\HH':=E^\perp_{\Ab}\subset \HH$. By Proposition~\ref{prop_hersubcat2},  $\HH'\subset \HH$ is thick, and by Proposition~\ref{prop_hersubcat1}, $\HH'$ is abelian hereditary. Also, by Proposition~\ref{prop_hersubcat2} $D^b(\HH')\cong E^\perp$ is a full subcategory of $D^b(\HH)$, hence $\HH'$ is $\Ext$-finite. Clearly, $\HH'$  is Noetherian. Now we prove that $\HH'$ has Serre duality in the form~\eqref{eq_SerreDual}. Consider semi-orthogonal decomposition
    $$D^b(\HH)=\langle D^b(\HH'), \langle E\rangle\rangle,$$
    let $\pi\colon D^b(\HH)\to D^b(\HH')$ be the left adjoint functor to the inclusion. Explicitly, for
    $F\in D^b(\HH)$ it is defined by the exact triangle
    \begin{equation}
    \label{eq_pi}
        E\otimes \Hom^{\bul}(E,F)\to F\to \pi(F)\to E\otimes \Hom^{\bul}(E,F)[1].
    \end{equation}
    It is known \cite[Prop. 3.7]{BK} that $D^b(\HH')$ has a Serre functor $S_{\HH'}$, and one has 
    $$S_{\HH'}^{-1}(F)\cong \pi\circ S_{\HH}^{-1}(F)$$
    for any $F\in D^b(\HH')$. It suffices to verify that $S_{\HH'}$ sends $\HH'$ to $\HH'[1]$, then one can put $\tau':=S_{\HH'}[-1]$. Recall that $S_{\HH}=\tau[1]$, so we are to check that $\pi$ sends $\HH$ to $\HH'$ (i.e., to complexes concentrated in degree $0$).
    Let $F\in \HH$, the long exact sequence of cohomology associated with \eqref{eq_pi} shows that 
    $H^i(\pi(F))=0$ except for possibly $i=0, -1$, and that 
    $$H^{-1}(\pi(F))=\ker(E\otimes \Hom(E,F)\to F).$$
    The canonical map $E\otimes \Hom(E,F)\to F$ is injective since $E$ is simple, so $H^{-1}(\pi(F))=0$ and $\pi(F)$ is isomorphic to an object of $\HH'$. Therefore, $\HH'$ satisfies (H3). Also note that
    \begin{equation}
        \label{eq_tau'}
        \tau'^{-1}\cong \pi\circ \tau^{-1}.
    \end{equation}

    For any point $y\ne x$ torsion sheaves in $\HH$ supported at $y$ belong to $\HH'$, so (H6) holds. Pick some line bundle $L$ in $\HH$, then \eqref{eq_pi} gives  the universal extension
    $$0\to L\to \pi(L)\to E\otimes \Ext^1(E,L)\to 0,$$
    hence $\pi(L)$ is a line bundle in $\HH'$. Let $0\ne V\in \HH'$ be a vector bundle. Pick a  point $y\ne x$ and a sphere-like sheaf $S_y$ supported at $y$. Using Lemma~\ref{lemma_homstorsion}, one constructs  an infinite sequence in~$\HH'$
    $$V=V_0\supset V_1\supset V_2\supset\ldots$$
    of decreasing sub-bundles with $\coker(V_{i+1}\to V_i)\cong S_y\in\HH'$. Thus $V$ has infinite length in $\HH'$, and
    $\HH'_0 = \HH'\cap \HH_0$: object in $\HH'$ has finite length in $\HH'$ if and only if it has finite length in $\HH$. So, (H4) holds. It is not hard to see that $\HH'/\HH'_0\cong \HH/\HH_0$. Hence, (H5) follows and notions of rank in $\HH$ and in $\HH'$ agree.

    Therefore, $\HH'$ satisfies Definition~\ref{def_HH}. 

    Using \eqref{eq_pi} and \eqref{eq_tau'},  we see that $\tau'$ preserves subcategory $\HH_y\cap \HH'$ for any point $y\in X$.    For $y\ne x$, we have $\HH_y\subset \HH'$ and $\tau=\tau'$ on $\HH_y$ , so there is one $\tau'$-orbit of simples in~$\HH_y$. For $y=x$, note that $\HH_x\cong \UU_r$ and by Proposition~\ref{prop_thicktubes} $\HH'\cap \HH_x\cong \UU_{r-1}$. Therefore $\tau'$ on $\HH_x\cap \HH'$ coincides with $\tau_{\UU_{r-1}}$, and there is one $\tau'$-orbit of $r-1$ simple objects in $\HH_x\cap \HH'$. It follows that points of $\HH'$ are the same as points of $\HH$,  the weight function of $\HH'$ is as stated, and that notions of support for $\HH'$ and for $\HH$ agree. 

    \medskip
    Now we treat the general case. We have $E=S^{[m]}$ for some simple sheaf $S$. One can check as in~\cite[Prop. 6.5]{Cheng} that 
    \begin{equation}
    \label{eq_SS1SS2}
        E^{\perp}_{\Ab}=\SS_1\times \SS_2, \quad\text{where}\quad \SS_1=\langle S,\tau S,\ldots, \tau^{m-1}S\rangle^\perp_{\Ab}, \SS_2=\langle \tau S,\ldots, \tau^{m-1}S\rangle_{\Ab}.
    \end{equation}
    Indeed, $\Hom^\bul(\SS_2,\SS_1)=0$ by definition and $\Hom^\bul(\SS_1,\SS_2)\cong \Hom^\bul(\tau^{-1}\SS_2,\SS_1)^*=0$ again by definition. Clearly $\Hom^\bul(E,\SS_1\times \SS_2)=0$ (note that the simple factors of $E$ are $S,\tau S,\ldots, \tau^{m-1}S$). Further,  $E$ together with $\SS_2$ generate $S$. Since  $S,\tau S,\ldots, \tau^{m-1}S$ is an exceptional collection, 
    one has 
    $$D^b(\HH)=\langle \langle S,\tau S,\ldots, \tau^{m-1}S\rangle^\perp, \langle S,\tau S,\ldots, \tau^{m-1}S\rangle\rangle = \langle\langle \SS_1\rangle, \langle \SS_2\rangle, \langle E\rangle\rangle=\langle\langle \SS_1\rangle\times \langle \SS_2\rangle, \langle E\rangle\rangle.$$ 
    Then~\eqref{eq_SS1SS2} follows since 
    $$E^\perp_{\Ab}=(\langle \SS_1\rangle\times \langle \SS_2\rangle)_{\Ab}=\langle \SS_1\rangle_{\Ab}\times\langle \SS_2\rangle_{\Ab}=\SS_1\times \SS_2.$$

    Family $ \tau S,\ldots, \tau^{m-1}S$ in $\SS_2$ is quiver-like with the $\Ext$-quiver $A_{m-1}$, by Corollary~\ref{cor_40} $\SS_2\cong \AA_{m-1}$. Remaining statements regarding $\SS_1=\langle S,\tau S,\ldots, \tau^{m-1}S\rangle^\perp_{\Ab}$ are obtained by applying inductively $m=1$ case, because $S,\tau S,\ldots, \tau^{m-1}S$ is an exceptional collection of simple sheaves. 
\end{proof}

\subsection{Weighted projective lines}
\label{section_WPL}

Here we recall briefly the construction of the category $\coh\X$ for a weighted projective line $\X=(\P^1,w)$, proposed by Geigle and Lenzing~\cite{GL}. We will need this description to provide some examples in Section~\ref{section_examples}.  

Let $x_1,\ldots,x_n$ be the weighted points of $\X$ and let $r_i=w(x_i)$. Let $V=H^0(\P^1,\O(1))$, choose non-zero $u_i\in V$ such that $u_i(x_i)=0$. Denote
$$A_\X:=S^\bul(V)\otimes \k[U_1,\ldots,U_n]/(U_1^{r_1}-u_1,\ldots,U_n^{r_n}-u_n),$$
where $S^\bul(V)$ denotes the symmetric algebra.
Let $\L$ be the abelian group generated by elements $\bar c,\bar x_1,\ldots,\bar x_n$ and relations $\bar c=r_i\cdot \bar x_i$ for $i=1,\ldots,n$. Then $A_\X$ is an $\L$-graded commutative algebra with the grading given by $\deg(U_i)=\bar x_i$, $\deg (v)=\bar c$ for $v\in V$.
Define $\coh\X$ as the Serre quotient
$$\coh\X=\frac{\mathrm{mod}^\L{-}A_\X}{\mathrm{mod}_0^\L{-}A_\X},$$
where $\mathrm{mod}^\L{-}A_\X$ is the category of finitely generated $\L$-graded $A_\X$-modules, and $\mathrm{mod}_0^\L{-}A_\X\subset \mathrm{mod}^\L{-}A_\X$ is the subcategory of finite-dimensional modules. 

For any $\bar x\in \L$, denote the object in $\coh \X$ given by the shifted free module $A_{\X}(\bar x)$ by $\O(\bar x)$. This is a line bundle, and all line bundles in $\coh\X$ have this form.

One computes $\Hom$ and $\Ext$ spaces between line bundles as follows. Any element in $\L$ can be written uniquely as
$$a\bar c+\sum_{i=1}^n b_i\bar x_i,\quad\text{where}\quad 0\le b_i\le r_i-1.$$
This presentation is called \emph{normal form}. One has
$$\dim\Hom(\O(\bar x),\O(\bar y))=\begin{cases}
    a+1, & a\ge 0,\\
    0, & a<0,
\end{cases}$$
where $a$ is the coefficient at $\bar c$ in the normal form of $\bar y-\bar x$.

Denote $\bar\omega=-2\bar c+\sum_{i=1}^n(r_i-1)\bar x_i$, then  Serre duality for $\coh\X$ is given by the shift of grading by $\bar\omega$:
$$\Ext^1(F_1,F_2)\cong \Hom(F_2,F_1(\bar\omega))^*.$$
This allows one to find $\Ext^1$ between line bundles easily.

Finally, remark that the functors $\s_{x}$ and $c_x \colon\coh\X\to\coh\X$ from Section~\ref{section_twist} are given by the shift of grading on graded modules:
$$\s_{x_i}(F)=F(\bar x_i),\quad \s_x(F)=F(\bar c)$$
for $i=1,\ldots,n$ and any non-singular point $x\in X$, while
$$c_x(F)=F(\bar c)$$
for all $x\in X$.

\section{Thick subcategories  on weighted projective curves, big and small}
\label{section_bigsmall}

Here we introduce two classes of thick subcategories on weighted projective curves that play special role. Similar results have been obtained in \cite{Cheng} for weighted projective \textbf{lines}.

\medskip
Let  $\X=(X,w)$ be a weighted projective curve. 
We  denote by $\coh_x\X\subset \coh \X$ the full subcategory of torsion sheaves supported at a point $x\in X$, and  by $\coh_0\X\subset \coh \X$ the full subcategory of all torsion sheaves.
We denote by $D^b_x(\coh\X)\subset D^b(\coh\X)$ the full subcategory of complexes with  cohomology supported  at $x$. By Proposition~\ref{prop_hersubcat2}, $D^b_x(\coh\X)\cong D^b(\coh_x\X)\cong D^b(\UU_{w(x)})$. 

\begin{definition}
\label{def_curvelike}
Let $\SS\subset \coh\X$ and $\TT\subset D^b(\coh\X)$ be the corresponding thick subcategories (recall Proposition~\ref{prop_hersubcat2}). We will say that
\begin{enumerate}
\item  $\SS$ and $\TT$ are  \emph{small} if 
    $\TT$ is generated by an exceptional collection of torsion sheaves;

\item $\SS$ and $\TT$ are  \emph{big} if 
$\TT$ is the orthogonal to an exceptional collection  of torsion sheaves;

\item $\SS$ and $\TT$ are  \emph{curve-like} if $\SS$ is the essential image of a fully faithful exact functor $\coh\X'\to \coh \X$
preserving rank and support of sheaves for some weighted projective curve  $\X'=(X,w')$.
\end{enumerate}
\end{definition}

\begin{remark}
    If $\TT\subset D^b(\coh\X)$ is curve-like and $\Phi\colon \coh\X'\to \coh \X$ is the functor from Definition~\ref{def_curvelike}, then the functor $D(\Phi)\colon D^b(\coh\X')\to D^b(\coh \X)$ is fully faithful with the image $\TT$ (by Proposition~\ref{prop_hersubcat2}).
\end{remark}

\begin{remark}
By Lemma~\ref{lemma_exceptionalatorbifold}, small subcategories are supported at orbifold points. By Proposition~\ref{prop_thicktubes}, any small subcategory is equivalent to a direct product of the form $\AA_{n_1}\times\ldots\times\AA_{n_k}$. By Proposition~\ref{prop_thicktubes} any subcategory of  a small category is also small, and any subcategory containing a big subcategory is big. 
Also note that small and big subcategories in $D^b(\coh\X)$ are admissible.
\end{remark}

\begin{remark}
By definition, there is a bijection between small and big subcategories of $D^b(\coh\X)$, given by $\TT\mapsto \TT^\perp$.
\end{remark}

\begin{remark}
    Note that big subcategories in $D^b(\coh\X)$ are in bijection with collections $(\TT_x)_x$, where $\TT_x\subset D^b(\coh_x\X)$ is a subcategory generated by an exceptional collection, and $x$ runs over orbifold points of $\X$ (a collection of subcategories $(\TT_x)_x$ corresponds to its right orthogonal, which is big). For any orbifold point $x\in X$ the number of subcategories in $D^b(\coh_x\X)\cong D^b(\UU_{w(x)})$ generated by an exceptional collection is finite, and by Proposition~\ref{prop_thicktubes} equals $\frac 12$  of the total number of thick subcategories in $\UU_{w(x)}$. 
    The number of thick subcategories in $\UU_n$ is $\binom{2n}{n}$ by \cite[Prop. 2.4.2]{Di}. Hence, the number of big subcategories in $D^b(\coh\X)$ is
    $$\prod_{x\colon w(x)\ge 2} \frac 12\binom{2w(x))}{w(x)}.$$
\end{remark} 

\begin{remark}
    Note that for a thick subcategory in $D^b(\coh\X)$  condition ``to be  curve-like'' is stronger than ``to be equivalent to some category $D^b(\coh \X')$ for a weighted projective curve $\X'$'' as Example~\ref{example_3333} below shows. On the contrary, any thick abelian subcategory in $\coh \X$ that is equivalent to some category $\coh \X'$ for a weighted projective curve $\X'$, is curve-like, as we will see in Corollary~\ref{cor_curvecurve}.    
\end{remark}

Now we give several equivalent definitions of big subcategories, highlighting their importance.
Some of the equivalences below appear in \cite[Prop. 6.2, Th. 6.3]{Cheng} in the case of a weighted projective line.
\begin{prop}
    \label{prop_bigbig}
    Let $\X$ be a weighted projective curve and $\SS\subset \coh\X$ be a thick subcategory. Then the following conditions are equivalent:
    \begin{enumerate}
        \item $\SS=\langle E_1,\ldots, E_n\rangle_{\Ab}^\perp$
        for some exceptional collection $E_1,\ldots,E_n$ of torsion sheaves (that is, $\SS$ is big);
        \item[(1')] $\SS=^\perp\langle E_1,\ldots, E_n\rangle_{\Ab}$
        for some exceptional collection $E_1,\ldots,E_n$ of torsion sheaves;
        \item $\SS=\SS_1\times \SS_2$, where $\SS_1$ is curve-like and $\SS_2$ is small;
        \item $\SS$ contains a curve-like subcategory in $\coh\X$;
        \item $\SS$ contains a non-zero vector bundle and a sphere-like torsion sheaf;
        \item $\SS$ contains a non-zero vector bundle and $\SS=c_x(\SS)$ for any point $x\in X$  (cf. Definition~\ref{def_cx});
        \item $\SS$ contains a non-zero vector bundle and $\SS=c_x(\SS)$ for some point $x\in X$.
    \end{enumerate}
\end{prop}
\begin{proof}
   (1) $\Longleftrightarrow$ (1') by Serre duality.

   (1) $\Longrightarrow$ (2) by iterated application of Proposition~\ref{prop_excort}.

   (2) $\Longrightarrow$ (3) is trivial.

   (3) $\Longrightarrow$ (4) is easy: let $\Phi\colon \coh \X'\to \coh\X$ be a fully faithful functor preserving rank and support of sheaves with  $\im(\Phi)=\SS$. Take any sphere-like torsion sheaf and vector bundle on $\X'$ and apply $\Phi$ to get  a sphere-like torsion sheaf and a vector bundle on $\X$ in~$\SS$. 

   (4) $\Longrightarrow$ (1) is contained in a separate Lemma~\ref{lemma_4to1} below due to its length, this is the less trivial  implication.

   (1) $\Longrightarrow$ (5): if $\SS$ consists only of torsion sheaves, then $D^b(\coh\X)=\langle \langle\SS\rangle,\langle
    E_1,\ldots,E_n\rangle\rangle$ also does, and we get a contradiction. For the second, note that $c_x(E_i)\cong E_i$ for any $i$ by Lemma~\ref{lemma_cx}. Consequently, $c_x$ preserves the orthogonal to $\langle E_1,\ldots,E_n\rangle$, which is $\SS$.

   (5) $\Longrightarrow$ (6) is trivial.

   (6) $\Longrightarrow$ (4): let $V\ne 0$ be a vector bundle in $\SS$.  Consider  exact sequence~\eqref{eq_ctop}
   $$0\to c_x^{-1}(V)\to V\to Q\to 0.$$
   By Lemma~\ref{lemma_topm}, $Q\cong \oplus_{i=1}^{\rank(V)} Q_i$, where $Q_i$ are indecomposable sheaves of length $w(x)$ supported at $x$. Hence $Q_i$ are sphere-like by Lemma~\ref{lemma_exceptionalatorbifold}. Note that $c_x^{-1}(V)\in\SS$ and thus $Q_i\in\SS$ as needed.
\end{proof}

\begin{lemma}
    \label{lemma_4to1}
Let $\X$ be a weighted projective curve and $\SS\subset \coh\X$ be a thick subcategory. Assume  $\SS$ contains a non-zero vector bundle and a sphere-like torsion sheaf. Then $\SS$ is big.
\end{lemma}
\begin{proof}
   We argue by induction in the total weight of orbifold points. The base case is $w=1$, that is, $\X=X$ is a smooth projective curve. We refer to \cite[Lemma 4.1]{EL} to see that $\SS=\coh X$ as soon as $\SS$ contains a torsion sheaf and a vector bundle. Hence, $\SS$ is big as the orthogonal to the empty exceptional collection. Now we consider the general case. We divided the argument into several steps.

   Step 1. We may assume that $\SS^\perp$ (or, equivalently, $^\perp \SS$) contains no simple exceptional sheaves. 
   Indeed, assume $E$ is a simple exceptional sheaf and $\SS\subset E^\perp_{\Ab}$. By Proposition~\ref{prop_excort}, $E^\perp_{\Ab}\cong \coh \X'$ for a weighted projective curve $\X'$ with smaller total weight of orbifold points. We observe that $\SS$ as a subcategory of $\coh\X'$ also contains a vector bundle and a sphere-like torsion sheaf. Hence by induction hypothesis, $\SS$ is the right orthogonal in $\coh\X'$ to an exceptional collection $E_1,\ldots,E_n$ of torsion sheaves on $\X'$. We deduce that~$\SS$ is the right orthogonal in $\coh \X$ to the exceptional collection $ E_1,\ldots,E_n,E$ of torsion sheaves on $\X$.

   Step 2. We may assume that $\SS$ contains no exceptional torsion sheaves. To check this, we find it more convenient to work with derived categories, let $\TT=\langle\SS\rangle\subset D^b(\coh\X)$. 
   Assume $E\in \TT$ is an exceptional torsion sheaf, supported at $x\in X$. 
   By Propositions~\ref{prop_excort} and~\ref{prop_hersubcat2}, $^\perp E\cong  D^b(\AA_{l-1})\times D^b(\coh\X')$, where $\X'$ is a weighted projective curve  with smaller total weight of orbifold points, $l$ is the length of $E$, and $D^b(\AA_{l-1})$ is generated by an exceptional collection of torsion sheaves supported at $x$.  Let $\TT':=\TT\cap ^\perp E=\TT_1\times \TT_2$, where $\TT_1\subset D^b(\AA_{l-1})$ and $\TT_2\subset D^b(\coh\X')$ are thick subcategories. 
   We have semi-orthogonal decompositions
   $$D^b(\coh\X)=\langle \langle E\rangle,^\perp E\rangle=\langle \langle E\rangle,D^b(\AA_{l-1}),D^b(\coh\X')\rangle\quad\text{and}\quad \TT=\langle \langle E\rangle,\TT_1,\TT_2\rangle.$$   
   We claim that $\TT_2$ as a subcategory of $D^b(\coh\X')$ also contains a vector bundle and a sphere-like torsion sheaf. For the former, assume $\TT_2$ is a torsion subcategory, then all 
   $\TT=\langle \langle E\rangle,\TT_1,\TT_2\rangle$ is a torsion subcategory, and contains no vector bundles.
   For the latter, assume~$\TT_2$ has no sphere-like torsion sheaves. 
   Then $\TT_2 \cap D^b_x(\coh\X)$ is generated by an exceptional collection (see Proposition~\ref{prop_thicktubes}), as well as $\TT_1$ (see Proposition~\ref{prop_thicklines}). It follows that 
   $$\TT\cap D^b_x(\coh\X)=\langle \langle E\rangle,\TT_1,\TT_2\rangle\cap D^b_x(\coh\X)=\langle \langle E\rangle,\TT_1,\TT_2\cap D^b_x(\coh\X)\rangle$$
   is also generated by an exceptional collection. Recall that $D^b_x(\coh\X)\cong D^b(\UU_{w(x)})$, Proposition~\ref{prop_thicktubes} implies now that  
   $\TT\cap \coh_x\X$ contains no sphere-like sheaves. Therefore, $\TT$ contains a sphere-like sheaf $M$ supported at some other point $y\in\X$, and clearly $M\in ^\perp E$ and $M\in \TT_2$. So, the claim holds.
   
   Applying induction hypothesis to $\TT_2\subset D^b(\coh\X')$, we see that $\TT_2$ is big:
   $$D^b(\coh\X')=\langle \TT_2 , \langle E_1,\ldots,E_n\rangle\rangle,$$
   where $E_1,\ldots,E_n$ is an  exceptional collection of torsion sheaves on $\X'$. By Proposition~\ref{prop_thicklines}, 
   $$D^b(\AA_{l-1})=\langle \TT_1 , \langle F_1,\ldots,F_m\rangle\rangle,$$
   where $F_1,\ldots,F_m$ is an  exceptional collection of torsion sheaves on $\X$.
   Therefore,
   \begin{multline*}
       D^b(\coh\X)=\langle \langle E\rangle,D^b(\AA_{l-1}),D^b(\coh\X')\rangle=\langle\langle E\rangle, \TT_1 , \langle F_1,\ldots,F_m\rangle, \TT_2, \langle E_1,\ldots,E_n\rangle\rangle=\\
       =\langle \langle E\rangle, \TT_1 , \TT_2,  \langle F_1,\ldots,F_m\rangle,\langle E_1,\ldots,E_n\rangle\rangle=
   \langle \TT, \langle F_1,\ldots,F_m, E_1,\ldots,E_n \rangle\rangle,
   \end{multline*}
   and $\TT$ is the right orthogonal to the exceptional collection $F_1,\ldots,F_m, E_1,\ldots,E_n$ of torsion sheaves. Hence $\SS=\TT_{\Ab}$ is big.

   Step 3. Let $x\in X$, we claim that $\SS\cap \coh_x\X$ is zero or generated by one sphere-like sheaf $M_x$. In the latter case, any sheaf in $\SS$ supported at $x$ is an iterated extension of 
   $M_x$. Also, for any $F\in\SS$ any non-zero $f\colon F\to M_x$ is surjective.
   
   Indeed, $\SS\cap \coh_x\X$ is a thick subcategory in $\coh_x\X\cong \UU_{w(x)}$ without exceptional objects by Step 2, and the claim follows from Proposition~\ref{prop_thicktubes}. For the last statement, note that $\im(f)$ belongs to $\SS$ since $\SS$ is thick. Hence  $\im(f)$ is an iterated extension of $M_x$ and must be all $M_x$.

   Step 4. We claim that the intersection $\SS\cap \coh_x\X$ is non-zero for any $x\in X$.
   Let $V\in\SS$ be a non-zero vector bundle with the minimal possible rank and $M\in \SS$ be a sphere-like torsion sheaf. One can build an infinite sequence
   $$V=V_0\supset V_1\supset V_2\supset\ldots,$$
   of subobjects fitting into exact sequences 
   \begin{equation}
   \label{eq_VVM}
       0\to V_{n+1}\to V_n\to M\to 0.
   \end{equation}
   Indeed, $\Hom(V_n,M)\ne 0$ by Lemma~\ref{lemma_homstorsion}, and any non-zero morphism $f\colon V_n\to M$ is surjective by Step 3. One can take $V_{n+1}=\ker f$. Note that all $V_i\in \SS$ since $\SS$ is thick. 
   
   Now we claim that $\dim\Hom(V_n,V)$ can be arbitrarily large for $n>>0$. Indeed,
   \begin{align*}
    \dim\Hom(V_n,V)&\ge \chi([V_n],[V])= &\\
    &=\chi([V_{n-1}],[V])-\chi([M],[V]) &\text{by \eqref{eq_VVM}}\\
    &=\chi([V_{n-1}],[V])+\chi([V],[\tau M]) &\text{by Serre duality}\\
    &=\chi([V_{n-1}],[V])+\dim\Hom(V,\tau M)= &\\
    &=\chi([V_{n-1}],[V])+\rank (V) &\text{by Lemma \ref{lemma_homstorsion}}\\
    &=\ldots &\\
    & = \chi([V],[V])+n\cdot \rank (V).
   \end{align*}
   Let $S$ be a simple sheaf supported at $x$ such that $\Hom(V,S)\ne 0$ (use Lemma~\ref{lemma_homstorsion}). Choose non-zero $f\colon V\to S$. Consider the map induced by $f$:
   $$\Hom(V_n,V)\to \Hom(V_n,S).$$
   Note that $\dim\Hom(V_n,V)$ can be arbitrarily large for big $n$, while $\dim\Hom(V_n,S)\le \rank V_n=\rank V$ by Lemma~\ref{lemma_homstorsion}. Hence for some $n$ there exist a  non-zero morphism $g\colon V_n\to V$ such that $fg=0$. Since $\SS$ is thick, $\ker(g)\in\SS$ and is a vector bundle, and by the minimality assumption, $\ker(g)$ must be zero. Let $F=\coker(g)$, this is a torsion sheaf since $\rank(V_n)=\rank(V)$, and $F\in\SS$. Clearly  $f$ factors through $F$, and so $x$ belongs to the support of $F$. Some direct summand of $F$ is a sheaf in $\SS\cap \coh_x\X$.
   $$\xymatrix{V_n\ar[rr]^g&& V\ar[rr] \ar[d]^f && F\ar@{-->}[lld] \\ && S &&}$$

   Step 5. Now we are going to show that $\X=X$ is a smooth curve and we can use the induction base. Let $x\in X$ be a point and $M$ be a sphere-like sheaf in $\SS$ supported at~$x$ from Step~4. Let $S\subset M$ be a simple subsheaf (note that such $S$ is unique). Assume $w(x)>1$, then $S$ is exceptional. By Step 1, $S$ is not in $\SS^\perp$. It means that $\Hom(F,S)\ne 0$ or $\Ext^1(F,S)\ne 0$ for some $F\in\SS$, and one can choose $F$ to be indecomposable. If $\Hom(F,S)\ne 0$, consider the composition $F\to S\to M$. Its image is $S$ and belongs to $\SS$ since $\SS$ is thick. This contradicts to Step 2. Now, if $\Ext^1(F,S)\ne 0$ then $F$ is a torsion sheaf supported at $x$, and is an iterated extension of $M$ by Step 3.  Hence $\Ext^1(M,S)\ne 0$. By Serre duality $\Hom(\tau^{-1}S,M)\ne 0$. But $M$ has unique simple subsheaf, and $S\not\cong \tau^{-1}S$, we get a contradiction. We conclude that $w(x)=1$.
\end{proof}

\section{Main results}
\label{section_main}

Throughout this section $\X=(X,w)$ is a weighted projective curve. 
Here we prove our main result (Theorem~\ref{theorem_main}) saying that any  thick subcategory in $\coh\X$ or $D^b(\coh\X)$ is either big or quiver-like, and obtain some consequences of this.

\subsection{Main theorem}
We start from the following observation.
\begin{prop}
\label{prop_curvelikequiverlike}
    Category $D^b(\coh\X)$ is quiver-like if and only if $X\cong \P^1$ and $\X$ has domestic type. On the contrary, the category $\coh\X$ is not a finite length category, therefore never is quiver-like. 
\end{prop}
\begin{proof}
    Note that $D^b(\coh \X)$ has a strong generator. Indeed, it has a semi-orthogonal decomposition 
    $\langle D^b(\coh X),\TT\rangle$ into the derived category of a smooth projective curve (which has a strong generator by \cite[Th. 3.1.4]{BvdB}) and a small subcategory $\TT$ (which is generated by an exceptional collection hence also has a strong generator).   
    Assume $D^b(\coh\X)\cong D^b_0(\k Q)$, then $D^b_0(\k Q)$ has a strong generator. By Proposition~\ref{prop_QLgeneral}, $Q$ is finite and acyclic. Hence $D^b_0(\k Q)\cong D^b(\modd\k Q)$, where $\k Q$ is a hereditary finite-dimensional algebra. It is known then that $\X$ is a weighted projective line of domestic type (see~\cite[Sect. 10]{Lenzing_hercat}).
\end{proof}

\begin{theorem}
\label{theorem_main}
Let $\X$ be a weighted projective curve, $\TT\subset D^b(\coh\X)$ be a thick subcategory, and $\SS=\TT\cap \coh\X$. Then at least one of the following holds:
\begin{itemize}
\item $\TT$ and $\SS$ are big, or 
\item $\TT$ and $\SS$ are  quiver-like. 
\end{itemize}
\end{theorem}
\begin{proof}
Note that $\TT$ is big $\Longleftrightarrow$ $\SS$ is big by definition.

Case 1: $\SS$ contains a non-zero vector bundle and a sphere-like torsion sheaf. Then $\SS$ is big by Lemma~\ref{lemma_4to1}.

Case 2: $\SS$ does not contain non-zero vector bundles, that is, $\SS$ consists only of torsion sheaves. Then $\TT$ and $\SS$ are quiver-like by Corollary~\ref{cor_39} applied to $\AA=\coh\X$ and the function $r([F])=\len(F)$, where $F$ is a torsion sheaf.

Case 3: $\SS$ does not contain sphere-like torsion sheaves. We prove then that $\TT$ and $\SS$ are quiver-like. Let  $E_1,\ldots, E_n$ be the maximal exceptional collection of torsion sheaves in $\TT$, then $\TT=\langle \langle E_1,\ldots, E_n\rangle,\TT'\rangle$ for $\TT':=^\perp\langle E_1,\ldots, E_n\rangle\cap \TT$. By our assumptions, $\TT'$ does not contain exceptional torsion sheaves and sphere-like torsion sheaves. It follows then (see Proposition~\ref{prop_thicktubes}) that~$\TT'$ does not contain torsion sheaves at all. Therefore, $\TT'$ is quiver-like by 
Corollary~\ref{cor_39} applied to $\AA=\coh\X$ and the function $r([F]):=\rank (F)$ (this argument, originating from~\cite{EL}, is the heart of the paper). Let $\{V_i\}_{i\in I}$ be a vertex-like generating  family of vector bundles in $\TT'$. Also, $\langle E_1,\ldots, E_n\rangle$ is quiver-like by Case 2, let $E'_1,\ldots,E'_n$ be a vertex-like generating family of torsion sheaves in $\langle E_1,\ldots, E_n\rangle$. Then $\{E'_i\}_{1\le i\le n}\cup \{V_i\}_{i\in I}$ is   a vertex-like generating family for $\TT=\langle \langle E_1,\ldots, E_n\rangle,\TT'\rangle$. Indeed, $\Hom(E'_i, V_j)=0$ since $E'_i$ is torsion and $V_j$ is torsion-free, and $\Hom(V_j,E'_i)=0$ by semi-orthogonality. Hence, $\TT$ and $\SS$ are quiver-like by Corollary~\ref{cor_40}.
\end{proof} 

Now we formulate some  consequences of Theorem~\ref{theorem_main} and examine the variety of thick subcategories.

\subsection{Admissible subcategories}

\begin{prop}
\label{prop_admexc}
\begin{enumerate}
    \item Let $\TT \subset D^b(\coh\X)$ be a quiver-like  admissible subcategory. Then $\TT$ is generated by an exceptional collection.

    \item Let $\TT\subset D^b(\coh\X)$ be an admissible category. Then $\TT$ is big or $\TT$ is generated by an exceptional collection. 
\end{enumerate}
 \end{prop}
\begin{proof}
    For (1), recall that  $D^b(\coh\X)$ has a strong generator, hence $\TT$ also has a strong generator. Assume $\TT\cong D^b_0(\k Q)$, then by Proposition~\ref{prop_QLgeneral} $Q$ is finite and acyclic. Therefore $D^b_0(\k Q)\cong D^b(\modd\k Q)$, and the latter category is generated by an exceptional collection (for example, of indecomposable projective $\k Q$-modules).

    For (2), suppose $\TT$ is not big, then $\TT$ is quiver-like by Theorem~\ref{theorem_main}, and part (1) applies.
\end{proof}

\begin{corollary}
\label{cor_admexcP1}
    Assume that $\X$ is a weighted projective line: that is, $X\cong \P^1$. Then any admissible subcategory  $\TT\subset D^b(\coh\X)$ is generated by an exceptional collection. Moreover, any exceptional collection $E_1,\ldots,E_n$ in $D^b(\coh \X)$ can be completed to a full one.
\end{corollary}
\begin{proof}
    By Theorem~\ref{theorem_main}, $\TT$ is big or quiver-like. If $\TT$ is big then $\TT\cong \TT_1\times\TT_2$ by Proposition~\ref{prop_bigbig}, where $\TT_1\cong D^b(\coh\Y)$ for some weighted projective line $\Y=(\P^1,w)$ and $\TT_2$ is small. Note that $\TT_1$ (see~\cite{GL}) and $\TT_2$ are generated by exceptional collections, hence $\TT$ also does. Alternatively, if $\TT$ is quiver-like then Proposition~\ref{prop_admexc}(1) applies.

    The second part follows by considering the orthogonal $\langle E_1,\ldots,E_n\rangle^\perp$ and applying the first part.
\end{proof}

\begin{corollary}
\label{cor_nosod}
    Assume that $\X$ is not a weighted projective line: that is, $g(X)\ge 1$. Then any admissible subcategory  $\TT\subset D^b(\coh\X)$ is either small or big. In particular, 
    there are no exceptional vector bundles on $\X$.
\end{corollary}
\begin{proof}
    One has a semi-orthogonal decomposition $D^b(\coh\X)=\langle\TT^\perp, \TT\rangle$. If $\TT$ is not small neither big then neither of $\TT,\TT^\perp$ is big, and both are generated by an exceptional collection by Proposition~\ref{prop_admexc}(2). Consequently, all $D^b(\coh\X)$ is generated by an exceptional collection, what is impossible for $g(X)\ge 1$ (for example, because $K_0(\coh\X)$ is not finitely generated).  

    For the second statement note that a subcategory generated by an exceptional  vector bundle is admissible and by the above is small or big and hence contains non-trivial torsion sheaves, which is not possible.
\end{proof}
\begin{remark}
\label{rem_okawa}
    Corollary~\ref{cor_nosod} provides an alternative way to prove a well-known fact: the derived category of a connected smooth projective curve $X$ of genus $>0$ has no non-trivial semi-orthogonal decompositions. Indeed, $D^b(\coh X)$ has no non-trivial small/big subcategories. Note that the classical argument by Okawa~\cite{Okawa} is based on the fact that the canonical line bundle on $X$  has enough non-vanishing sections, and we make no use of that.
\end{remark}

\begin{corollary}
\label{cor_curvecurve}
    Let $\SS\subset \coh\X$ be a thick subcategory and $\SS\cong \coh \Y$ for some weighted projective curve $\Y=(Y,w')$. Then $\SS$ is curve-like.
\end{corollary}
\begin{proof}
    By Theorem~\ref{theorem_main}, $\SS$ is either big or quiver-like. Since $\coh \Y$ is not quiver-like (Proposition~\ref{prop_curvelikequiverlike}), $\SS$ is big. By Proposition~\ref{prop_bigbig}, $\SS\cong \SS_1\times\SS_2$, where $\SS_1$ is curve-like and $\SS_2$ is small. Since $\coh \Y$ is indecomposable into a direct product, we have $\SS_2=0$ and $\SS$ is curve-like.
\end{proof}
\begin{remark}
    Note that  the analogue of Corollary~\ref{cor_curvecurve} for derived categories does not hold. See Example~\ref{example_3333} for a thick subcategory $\TT\subset D^b(\coh \X)$, which is equivalent to $D^b(\coh\Y)$ for a weighted projective line $\Y$, but is not curve-like.
\end{remark}

\subsection{The Jordan--H\"older property and  absence of phantoms}
We observe next that the derived categories of weighted projective curves do not demonstrate such pathologies as phantom subcategories or violation of the Jordan--H\"older property. 

Recall the a semi-orthogonal decomposition is called \emph{maximal} if its components are semi-orthogonal indecomposable. A triangulated category $\DD$ is said to satisfy the Jordan--H\"older property if for any two maximal semi-orthogonal decompositions $\DD=\langle\TT_1,\ldots,\TT_n\rangle=\langle\TT'_1,\ldots,\TT'_m\rangle$ one has $n=m$ and $\TT_i\cong \TT'_{\sigma(i)}$ for some permutation $\sigma\in S_n$ and all $i$.
\begin{corollary}
\label{cor_JH}
    Category $D^b(\coh \X)$ satisfies the Jordan--H\"older property. More precisely, any maximal semi-orthogonal decomposition of $D^b(\coh \X)$ 
    \begin{itemize}
        \item is generated by an exceptional collection if $X\cong \P^1$, or
        \item contains one copy of $D^b(\coh X)$, while other components are generated by exceptional objects if $g(X)\ge 1$. 
    \end{itemize}
\end{corollary}
\begin{proof}
By Proposition~\ref{prop_admexc}(2), an indecomposable  admissible subcategory $\TT$ in $D^b(\coh\X)$ is big or is generated by an exceptional object. In the first case $\TT$ is equivalent to $D^b(\coh \X')$  where $\X'=(X,w')$ (see Proposition~\ref{prop_bigbig}), and since $\TT$ is indecomposable we have $w'=1$, $g(X)\ge 1$, and $\TT\cong D^b(\coh X)$. 
\end{proof}

Recall that a (smooth, proper) triangulated category $\TT$ is called a \emph{phantom} if $K_0(\TT)=0$ while $\TT\ne 0$.
\begin{corollary}
\label{cor_phantom}
The category $D^b(\coh\X)$ does not contain  a full triangulated subcategory which is a phantom.
\end{corollary}
\begin{proof}
    Follows from Theorem~\ref{theorem_main}. Indeed, assume $\TT\subset D^b(\coh \X)$ is non-zero. If $\TT\cong D^b_0(Q)$ is quiver-like then $K_0(\TT)\cong \Z^{|Q_0|}\ne 0$ (see Proposition~\ref{prop_pleasant}), and if $\TT$ is big then $K_0(\TT)\supset K_0(X)\ne 0$ (see Proposition~\ref{prop_bigbig}).
\end{proof}

\subsection{Torsion and torsion-free subcategories}
Finally, we discuss torsion-free subcategories in $\coh\X$.
\begin{definition}
    Let $\X$ be a weighted projective curve and $\SS\subset \coh\X$ be a thick subcategory. We will say that $\SS$ is
    \begin{itemize}
        \item \emph{torsion} if all sheaves in $\SS$ are torsion;
        \item \emph{torsion-free} if all torsion sheaves in $\SS$ are zero;
        \item \emph{mixed} if $\SS$ contains both non-zero torsion and torsion-free sheaves.
    \end{itemize}
    We will use the same terminology for the corresponding subcategories in $D^b(\coh\X)$.
\end{definition}

It follows from Theorem~\ref{theorem_main} that any torsion and any torsion-free subcategory is quiver-like. It follows from definitions  that any small subcategory is torsion and any big subcategory is mixed. In particular, any small or big subcategory contains non-trivial torsion sheaves.

As the following proposition shows,  a mixed subcategory is  made of  a torsion and a torsion-free subcategory.
\begin{prop}
\label{prop_ttf}
    Let $\TT\subset D^b(\coh\X)$ be a thick subcategory. Assume $\TT$ is not big, then $\TT=\langle\TT_1,\TT_2\rangle$, where $\TT_1$ is torsion and $\TT_2$ is torsion-free. Semi-orthogonal decomposition of $\TT$ with such properties is unique. Moreover, if $\TT_2\ne 0$ then $\TT_1$ is small.
\end{prop}
\begin{proof}
    By Theorem~\ref{theorem_main}, $\TT$ and $\TT\cap \coh\X$ are quiver-like. Let $F_i, i\in I$ be the vertex-like family in $\coh \X$, generating $\TT$. 
    Denote by $I_1$ (resp. $I_2$) the set of $i\in I$ such that $F_i$ is torsion (resp. torsion-free).     
    Let $\TT_1=\langle F_i\rangle_{i\in I_1}$, $\TT_2=\langle F_i\rangle_{i\in I_2}$.  Clearly, $\TT_1$ and $\TT_2$ together generate~$\TT$. Also, for any $i_1\in I_1, i_2\in I_2$ one has $\Hom(F_{i_2},F_{i_1})=0$ by the definition of vertex-like and  $\Ext^1(F_{i_2},F_{i_1})=0$
    since $F_{i_2}$ is torsion-free and $F_{i_1}$ is torsion. Therefore, $\TT_1\subset \TT_2^\perp$ and we get a semi-orthogonal decomposition $\TT=\langle\TT_1,\TT_2\rangle$.
    Note that $\TT_1\cap \coh\X$ coincides with the family of all torsion sheaves in $\TT$, and the uniqueness follows. 

    For the second part, assume $\TT_2\ne 0$. If $\TT_1$ contains a sphere-like torsion sheaf then $\TT$ is big by Lemma~\ref{lemma_4to1}, which contradicts to assumptions. If $\TT_1$ has no sphere-like torsion sheaves then $\TT_1$ is generated by an exceptional collection of sheaves (see Proposition~\ref{prop_thicktubes}), hence is small.
\end{proof}

\section{Examples}
\label{section_examples}

In this section we provide some examples of quiver-like subcategories, illustrating general results from previous sections and demonstrating variety of the world of thick subcategories, cf. Figure~\ref{fig_1}.

Let $\X=(\P^1,w)$ be a weighted projective line with weighted points $x_1,\ldots,x_n$ of weights $r_1,\ldots,r_n$. We will use notation from Section~\ref{section_WPL} throughout this section.
Category $D^b(\coh\X)$ has a full exception collection of line bundles
\begin{equation}
\label{eq_OOO}
\{\O,\O(b \bar x_i), \O(\bar c)\}_{i=1\ldots n, b=1\ldots r_i-1}.
\end{equation}
This collection is strong: all $\Ext^i$ spaces between its objects vanish for $i\ne 0$. The $\Hom$ algebra of \eqref{eq_OOO} is generated by the arrows in the quiver~\eqref{eq_canonical} and the relations  coming from the equalities $U_i^{r_i}=u_i\in\Hom(\O,\O(\bar c))=V$.
\begin{equation}
\label{eq_canonical}
\xymatrix{ && \O(\bar x_1)\ar[r]^{U_1}& \O(2\bar x_1)\ar[r]^-{U_1} & \ldots \ar[r]^-{U_1}& \O((r_1-1)\bar x_1)\ar[rrd]^{U_1} && \\
\O \ar[rru]^{U_1}\ar[rr]^{U_2}\ar[rrdd]^{U_n} && \O(\bar x_2)\ar[r]^{U_2}& \O(2\bar x_2)\ar[r]^-{U_2} & \ldots \ar[r]^-{U_2}& \O((r_2-1)\bar x_2)\ar[rr]^{U_2} && \O(\bar c)\\
&&&\ldots&\ldots\\
&& \O(\bar x_n)\ar[r]^{U_n}& \O(2\bar x_n)\ar[r]^-{U_n} & \ldots \ar[r]^-{U_n} & \O((r_n-1)\bar x_n)\ar[rruu]^{U_n}&&
}
\end{equation}

\begin{example}
\label{example_star}
Fix some numbers $b_i\in\Z$, $0\le b_i\le r_i-1$. Consider the subcollection  
\begin{equation}
\label{eq_OOO1}
\{\O,\O(b \bar x_i)\}_{i=1\ldots n, b=1\ldots b_i}.
\end{equation}
of~\eqref{eq_OOO}. 
Let $\TT_l$ be the subcategory in $D^b(\coh\X)$ generated by~\eqref{eq_OOO1}, and $\SS_l=\TT\cap\coh\X$. Endomorphism algebra of~\eqref{eq_OOO1} is the path algebra of the quiver
\begin{equation}
\label{eq_star1}
 \xymatrix{ &&& \O(\bar x_1)\ar[r]&  \ldots \ar[r]& \O(b_1\bar x_1)  \\
  Q_l\colon &&\O \ar[ru]\ar[r]\ar[rd] &&\ldots&\\
  &&& \O(\bar x_n)\ar[r]&  \ldots \ar[r] & \O(b_n\bar x_n)}\\
\end{equation}
The vertex-like collection in $\SS_l$ generating $\TT_l$ is the left dual exceptional collection to~\eqref{eq_OOO1} (see~\cite{Bondal_associative} for the definition and details):
\begin{equation}
\label{eq_dualstar1}
\xymatrix{  S_{1,b_1}\ar[r]& S_{1,b_1-1}\ar[r]&  \ldots \ar[r]& S_{1,1}\ar[rd] & \\
&& \ldots &\ar[r]& \O\\
 S_{n,b_n}\ar[r]&  S_{n,b_n-1}\ar[r]& \ldots \ar[r] & S_{n,1}\ar[ru]&
}
\end{equation}
where $S_{i,j}:=\coker(\O((j-1)\bar x_i)\to \O(j\bar x_i))$ is a simple torsion sheaf supported at $x_i$.
Note that $S_{ij}\cong S_{i,j+r_i}$ and $\tau S_{ij}\cong S_{i,j-1}$.
Arrows in \eqref{eq_dualstar1} denote $\Ext^1$ spaces, hence the $\Ext$-quiver of~\eqref{eq_dualstar1} is $Q_l$, and $\SS_l\cong \moddo \k Q_l\cong \modd \k Q_l$. 
Note that~$\SS_l$ is in general a mixed subcategory.

Now consider the remaining part of~\eqref{eq_canonical}: 
\begin{equation}
\label{eq_OOO2}
\{\O(b \bar x_i), \O(\bar c)\}_{i=1\ldots n, b=b_i+1\ldots r_i-1}.
\end{equation}
Let $\TT_r$ be the subcategory in $D^b(\coh\X)$ generated by~\eqref{eq_OOO2}, and $\SS_r=\TT_r\cap\coh\X$. Endomorphism algebra of~\eqref{eq_OOO2} is the path algebra of the quiver 
\begin{equation}
\label{eq_star2}
\xymatrix{ &&  \O((b_1+1)\bar x_1)\ar[r]&  \ldots \ar[r]& \O((r_1-1)\bar x_1)\ar[rd] & \\
Q_r\colon &&&\ldots & \ar[r]& \O(\bar c)\\
 && \O((b_n+1)\bar x_n)\ar[r]&  \ldots \ar[r] & \O((r_n-1)\bar x_n)\ar[ru]&}
\end{equation}
The vertex-like collection in $\SS_r$ generating $\TT_r$ is obtained from the right dual exceptional collection to~\eqref{eq_OOO2} by mutating $\O(\bar c)$ to the right end:
\begin{equation}
\label{eq_dualstar2}
\xymatrix{  S_{1,r_1}\ar[r]& S_{1,r_1-1}\ar[r]&  \ldots \ar[r]& S_{1,b_1+2}\ar[rd] & \\
&& \ldots &\ar[r]& \O(\bar c-\sum (r_i-b_i-1)\bar x_i)\\
 S_{n,r_n}\ar[r]&  S_{n,r_n-1}\ar[r]& \ldots \ar[r] & S_{n,b_n+2}\ar[ru]&
}
\end{equation}
The $\Ext$-quiver of~\eqref{eq_dualstar2} is $Q_r^{\op}$ (and not $Q_r$), hence $\SS_r\cong \moddo \k Q_r^{\op}\cong \modd \k Q_r^{\op}$. 
Subcategory~$\SS_r$ is also mixed.       
\end{example}

\begin{example}
\label{example_3333}
In Example~\ref{example_star} let $n=4$, $r_1=\ldots=r_4=3$, $b_1=\ldots=b_4=1$. Then quivers $Q_1$ and $Q_2^{\op}$ are affine Dynkin quivers of type $\~D_4$. Categories  $\TT_l$ and $\TT_r$ are equivalent to $D^b_0(\k \~D_4)\cong D^b(\modd \k \~D_4)\cong D^b(\coh \X_{2,2,2})$, where $\X_{2,2,2}$ is a domestic weighted projective line of type $2,2,2$. Hence both components of semi-orthogonal decomposition
$$D^b(\coh \X)=\langle\TT_l,\TT_r\rangle$$
are equivalent to the derived category of a weighted projective line, and are not curve-like.
\end{example}

\begin{example}
\label{example_tf}
Assume that $n\ge 2$.
In collection~\eqref{eq_OOO}, make mutation of $\O(\bar c)$ to the left end. One gets an exceptional collection
\begin{equation}
\label{eq_OOOskew}
\{\O(-\bar c+\sum_{i=1}^n(r_i-1)\bar x_i),\O,\O(b \bar x_i)\}_{i=1\ldots n, b=1\ldots r_i-1}.
\end{equation}
Denote $\LL:=\O(-\bar c+\sum_{i=1}^n(r_i-1)\bar x_i)$, then $(\LL,\O)$ is an exceptional pair.
One has $$\Hom(\LL,\O)\cong \Hom(\O,\O((1-n)\bar c+\sum_i \bar x_i))=0$$ and 
$$\Ext^1(\LL,\O)\cong \Hom(\O,\LL(\bar \omega))^*\cong \Hom(\O,\O((n-3)\bar c+\sum_i (r_i-2)\bar x_i))^*\cong \k^{n-2}.$$
Therefore, pair $(\LL,\O)$ is vertex-like and its $\Ext$-quiver is the $m$-Kronecker quiver $\mathrm{K}_m$, where $m=n-2$. Hence $\TT=\langle \LL,\O\rangle\subset D^b(\coh\X)$ is equivalent to $D^b_0(\k \mathrm{K}_m)$, this subcategory is generated by an exceptional collection and therefore admissible. Also $\TT$ is clearly torsion-free. If $n=4$ then $m=2$ and $\TT$ is equivalent to $D^b(\coh\P^1)$ but is not curve-like.  

Suppose in addition that all $r_i=2$. Then 
$^\perp\TT$ is generated by completely orthogonal collection of line bundles $\O(\bar x_1),\ldots,\O(\bar x_n)$ and hence is torsion-free. In this case one has a semi-orthogonal decomposition of $D^b(\coh \X)$ in two torsion-free subcategories, and one of them  is equivalent to $D^b(\coh\P^1)$.
\end{example}

\begin{example}
    Let $X$ be a connected smooth projective curve of genus $g\ge 1$. Let $L$ be the family of all line bundles of degree $0$ on $X$. Clearly, $L$ is a vertex-like family. By Riemann-Roch Theorem, for any $\LL_1,\LL_2\in L$
    \begin{equation*}
    \dim\Ext^1(\LL_1,\LL_2)=
    \begin{cases} g-1 & \LL_1\not\cong \LL_2,\\
    g & \LL_1\cong \LL_2.
    \end{cases}    
    \end{equation*}
    By Corollary~\ref{cor_40}, one has equivalences
    $$\langle L\rangle\cong D^b_0(\k Q),\quad \langle L\rangle_{\Ab}\cong \moddo\k Q,$$
    where $Q$ is the following quiver: its set of vertices has the cardinality of $L$ (and of $\k$), any vertex has $g$ loops and any two distinct vertices are connected by $g-1$ arrows in both directions.    
\end{example}

\bibliographystyle{plain}
\bibliography{biblio_main.bib}

@article{EL,
title = {Thick subcategories on curves},
journal = {Advances in Mathematics},
volume = {378},
pages = {107525},
year = {2021},
issn = {0001-8708},
doi = {https://doi.org/10.1016/j.aim.2020.107525},
url = {https://www.sciencedirect.com/science/article/pii/S0001870820305533},
author = {Alexey Elagin and Valery A. Lunts}
}

@InProceedings{GL,
author="Geigle, Werner
and Lenzing, Helmut",
editors="Greuel, Gert-Martin
and Trautmann, G{\"u}nther",
title="A class of weighted projective curves arising in representation theory of finite dimensional algebras",
booktitle="Singularities, Representation of Algebras, and Vector Bundles",
year="1987",
publisher="Springer Berlin Heidelberg",
address="Berlin, Heidelberg",
pages="265-297",
isbn="978-3-540-47851-5"
}

@article{RvdB,
  title={Noetherian hereditary abelian categories satisfying {S}erre duality},
  author={Idun Reiten and Michel {van den Bergh}},
  journal={Journal of the American Mathematical Society},
  year={2002},
  volume={15},
  pages={295-366},
  url={https://api.semanticscholar.org/CorpusID:17119935}
}

@article{Hovey,
  title={Classifying subcategories of modules},
  author={Mark Hovey},
  journal={Transactions of the American Mathematical Society},
  year={1999},
  volume={353},
  pages={3181-3191},
  url={https://api.semanticscholar.org/CorpusID:17949137}
}

@article{Cheng,
title = {Wide subcategories of a domestic weighted projective line},
journal = {Journal of Pure and Applied Algebra},
volume = {228},
number = {9},
pages = {107669},
year = {2024},
issn = {0022-4049},
doi = {https://doi.org/10.1016/j.jpaa.2024.107669},
url = {https://www.sciencedirect.com/science/article/pii/S0022404924000665},
author = {Yiyu Cheng},
}

@phdthesis{Di,
  title={Thick subcategories for quiver representations},
  author={{Nikolay Dimitrov} Dichev},
  year={2009},
  school={University of Paderborn},
  url={https://api.semanticscholar.org/CorpusID:27930840}
}

@book{Kr, 
place={Cambridge}, 
series={Cambridge Studies in Advanced Mathematics}, 
title={Homological Theory of Representations}, 
publisher={Cambridge University Press}, 
author={Krause, Henning}, 
year={2021}, 
collection={Cambridge Studies in Advanced Mathematics}
}

@article{Bruning,
  title={Thick subcategories of the derived category of a hereditary algebra},
  author={Kristian Br{\"u}ning},
  journal={Homology, Homotopy and Applications},
  year={2007},
  volume={9},
  pages={165-176},
  url={https://api.semanticscholar.org/CorpusID:55453009}
}

@phdthesis{LH,
  title={Sur les $A_\infty$-cat\'egories},
  author={Kenji Lef{\`e}vre-Hasegawa},
  year={2003},
  school={Universit{\'e} Paris 7},
  url={https://api.semanticscholar.org/CorpusID:125075359}
}

@InProceedings{Lenzing,
author="Lenzing, Helmut",
title="Weighted projective lines and {R}iemann surfaces",
booktitle="Proceedings of the 49th Symposium on Ring Theory and Representation Theory",
year="2017",
publisher="Symp. Ring Theory Represent. Theory Organ. Comm.",
pages="67-79"
}

@inproceedings{Lenzing_fda_sing,
  title={Representations of finite-dimensional algebras and singularity theory},
  author={Lenzing, Helmut},
  booktitle={Trends in ring theory (Miskolc, 1996)},
  collection={CMS Conf. Proc.},
  volume={22},
  pages={71-97},
  year={1998},
  publisher={Amer. Math. Soc. Providence, RI},
  editors={Vlastimil Dlab, and L{\'a}szl{\'o} M{\'a}rki}
}

@article{AZ,
  title={Noncommutative Projective Schemes},
  author={Michael Artin and {James J.} Zhang},
  journal={Advances in Mathematics},
  year={1994},
  volume={109},
  pages={228-287},
  url={https://api.semanticscholar.org/CorpusID:122595833}
}

@article{LenzingReiten,
  title={Hereditary noetherian categories of positive {E}uler characteristic},
  author={Helmut Lenzing and Idun Reiten},
  journal={Mathematische Zeitschrift},
  year={2006},
  volume={254},
  pages={133-171},
  url={https://api.semanticscholar.org/CorpusID:120088367}
}

@inproceedings{LenzingFuchs,
  title={The algebraic theory of fuchsian singularties},
  author={Helmut Lenzing},
  booktitle={Advances in Representation Theory of Algebras},
  series={Contemporary Mathematics},
  volume={761},
  year={2021},
  editors={Ibrahim Assem and Christof Geiss and Sonia Trepode},
  pages={171-190},
  url={https://api.semanticscholar.org/CorpusID:202718855},
  publisher={AMS Press Inc.}
}

@article{ChenKrause,
  title={Introduction to coherent sheaves on weighted projective lines},
  author={Xiao-Wu Chen and Henning Krause},
  journal={arXiv:0911.4473v3 [math.RT]},
  year={2009},
  url={https://arxiv.org/abs/0911.4473v3}
}

@article{Krause_strings,
  title={The category of finite strings},
  author={Henning Krause},
  journal={Algebraic Combinatorics},
  year={2023},
  number={3},
  volume={6},
  pages={661-676},
  doi={10.5802/alco.274}
}

@article{IT,
  title={Noncrossing partitions and representations of quivers},
  author={Colin Ingalls and Hugh Thomas},
  journal={Compositio Mathematica},
  year={2009},
  volume={145},
  pages={1533-1562},
  url={https://api.semanticscholar.org/CorpusID:7808808}
}

@article{GL_perpen,
  title={Perpendicular Categories with Applications to Representations and Sheaves},
  author={Werner Geigle and Helmut Lenzing},
  journal={Journal of Algebra},
  year={1991},
  volume={144},
  pages={273-343},
  url={https://api.semanticscholar.org/CorpusID:123032434}
}

@article{Meltzer,
  title={Tubular mutations},
  author={Hagen Meltzer},
  journal={Colloquium Mathematicum},
  year={1997},
  number={2},
  volume={74},
  pages={267-274},
}

@article{Logvinenko,
  title={Spherical DG-functors},
  author={Rina Anno and Timothy Logvinenko},
  journal={Journal of European Mathematical  Society},
  year={2017},
  number={9},
  volume={19},
  pages={2577-2656},
  doi={10.4171/JEMS/724}
}

@article{BK,
  title={Representable functors, {S}erre functors, and mutations},
  author={Alexey Bondal and Mikhail Kapranov},
  journal={Mathematics of The USSR --- Izvestiya},
  year={1990},
  volume={35},
  issue={3},
  pages={519-541},
  url={https://api.semanticscholar.org/CorpusID:122293207}
}

@article{Bondal_associative,
  title={Representation of associative algebras and coherent sheaves},
  author={Alexey Bondal},
  journal={Mathematics of The USSR --- Izvestiya},
  year={1990},
  volume={34},
  issue={1},
  pages={23-42}
}

@book{Huyb,
  title={Fourier-Mukai transforms in algebraic geometry},
  author={Daniel Huybrechts},
  year={2006},
  publisher={Oxford University Press},
  place={Oxford},
  doi={10.1093/acprof:oso/9780199296866.001.0001}
}

@book{GelfandManin,
  title={Methods of Homological Algebra},
  author={S. I. Gelfand and Yu. I. Manin},
  publisher={Springer},
  series={Springer Monographs in Mathematics},
  year={1996},
  url={https://api.semanticscholar.org/CorpusID:117771187}
}

@book{Neeman,
    author = {Amnon Neeman},
    title = {Triangulated Categories},
    publisher = {Princeton University Press},
    year = {2001},
    volume={148},
    series={Annals of Mathematics Studies},
    doi={10.1515/9781400837212}
}

@article{KuzPer,
  title={Serre functors and dimensions of residual categories},
  author={Alexander Kuznetsov and Alexander Perry},
  journal={arXiv:2109.02026v2 [math.AG]},
  year={2021},
  url={https://arxiv.org/abs/2109.02026v2}
}

@article{BvdB,
author = {Bondal, Alexey and {Van den Bergh}, Michel},
year = {2003},
pages = {1-36},
title = {Generators and Representability of Functors in Commutative and Noncommutative Geometry},
volume = {3},
journal = {Moscow Mathematical Journal},
doi = {10.17323/1609-4514-2003-3-1-1-36}
}

@article{Keller_DerivingDG,
  title={Deriving DG categories},
  author={Bernhard Keller},
  journal={Annales Scientifiques de L`\'Ecole Normale Sup\'erieure},
  year={1994},
  volume={27},
  pages={63-102},
  url={https://api.semanticscholar.org/CorpusID:6063950}
}

@article{BondalLarsenLunts,
    author = {Bondal, Alexey I. and Larsen, Michael and Lunts, Valery A.},
    title = {Grothendieck ring of pretriangulated categories},
    journal = {International Mathematics Research Notices},
    volume = {2004},
    number = {29},
    pages = {1461-1495},
    year = {2004},
    issn = {1073-7928},
    doi = {10.1155/S1073792804140385},
    url = {https://doi.org/10.1155/S1073792804140385},
    eprint = {https://academic.oup.com/imrn/article-pdf/2004/29/1461/2253237/2004-29-1461.pdf},
}

@article{Keller_Ainfinity,
  title={Introduction to {A}-infinity algebras and modules},
  author={Bernhard Keller},
  journal={Homology, Homotopy and Applications},
  year={1999},
  volume={3},
  pages={1-35},
  url={https://api.semanticscholar.org/CorpusID:14705762}
}

@article{Rouquier,
  title={Dimensions of triangulated categories},
  author={Rapha{\"e}l Rouquier},
  journal={Journal of K-theory},
  year={2008},
  volume={1},
  pages={193-256},
}

@incollection{Lenzing_hercat, 
place={Cambridge}, 
series={London Mathematical Society Lecture Note Series}, 
title={Hereditary categories}, 
booktitle={Handbook of Tilting Theory}, 
publisher={Cambridge University Press}, 
author={Lenzing, Helmut}, 
editors={Angeleri H{\"u}gel, Lidia and Happel, Dieter and Krause, Henning}, 
year={2007}, 
pages={105–146}
}

@article{Okawa,
  title={Semi-orthogonal decomposability of the derived category of a curve},
  author={Shinnosuke Okawa},
  journal={Advances in Mathematics},
  year={2011},
  volume={228},
  pages={2869-2873},
  url={https://api.semanticscholar.org/CorpusID:119676659}
}

@article{Cadman,
  title={Using stacks to impose tangency conditions on curves},
  author={Charles Cadman},
  journal={American Journal of Mathematics},
  year={2007},
  volume={129},
  pages={405-427},
  url={https://api.semanticscholar.org/CorpusID:10323243}
}

@article{Parthasarathy,
  title={t-structures in the derived category of representations of quivers},
  author={Rajagopalan Parthasarathy},
  journal={Proceedings of the Indian Academy of Sciences --- Mathematical Sciences},
  year={1988},
  volume={98},
  pages={187-214},
  url={https://api.semanticscholar.org/CorpusID:56389139}
}

@book{Happel, 
place={Cambridge}, 
series={London Mathematical Society Lecture Note Series}, 
title={Triangulated Categories in the Representation of Finite Dimensional Algebras}, 
publisher={Cambridge University Press}, 
author={Happel, Dieter}, 
year={1988}, 
collection={London Mathematical Society Lecture Note Series}
}

@article{MiyachiYekutieli, 
title={Derived {P}icard Groups of Finite-Dimensional Hereditary Algebras}, 
volume={129}, 
DOI={10.1023/A:1012579131516}, 
number={3}, 
journal={Compositio Mathematica}, 
author={Miyachi, Jun-Ichi and Yekutieli, Amnon}, 
year={2001}, 
pages={341–368}
}

@article{Kuznetsov_JordanHolder,
  title={A simple counterexample to the {Jordan--H\"older} property for derived categories},
  author={Alexander Kuznetsov},
  journal={arXiv:1304.0903 [math.AG]  },
  year={2013},
  url={https://arxiv.org/abs/1304.0903}
}

@article{BohningBothmerSosna,
title = {On the {Jordan-–H\"older} property for geometric derived categories},
journal = {Advances in Mathematics},
volume = {256},
pages = {479-492},
year = {2014},
issn = {0001-8708},
doi = {https://doi.org/10.1016/j.aim.2014.02.016},
author = {Christian B\"ohning and Hans-Christian {Graf von Bothmer} and Pawel Sosna}
}

@article{Pirozhkov_delPezzo,
title = {Admissible subcategories of del {Pezzo} surfaces},
journal = {Advances in Mathematics},
volume = {424},
pages = {109046},
year = {2023},
issn = {0001-8708},
doi = {https://doi.org/10.1016/j.aim.2023.109046},
author = {Dmitrii Pirozhkov}
}

@article{GorchinskiyOrlov,
  title={Geometric Phantom Categories},
  author={Sergey Gorchinskiy and Dmitri Orlov},
  journal={Publications math{\'e}matiques de l'IH{\'E}S},
  year={2013},
  volume={117},
  pages={329 - 349},
}

@article{Bohning+,
  title={Determinantal {Barlow} surfaces and phantom categories},
  author={Christian B\"ohning and Hans-Christian Graf von Bothmer and Ludmil Katzarkov and Pawel Sosna},
  journal={Journal of the European Mathematical Society},
  year={2015},
  volume={17},
  pages={1569-1592}
}

@article {Chin,
    AUTHOR = {Chin, William},
     TITLE = {Hereditary and path coalgebras},
   JOURNAL = {Comm. Algebra},
  FJOURNAL = {Communications in Algebra},
    VOLUME = {30},
      YEAR = {2002},
    NUMBER = {4},
     PAGES = {1829--1831},
      ISSN = {0092-7872,1532-4125},
   MRCLASS = {16W30},
  MRNUMBER = {1894047},
       DOI = {10.1081/AGB-120013219},
       URL = {https://doi.org/10.1081/AGB-120013219},
}

@article {Simson,
    AUTHOR = {Simson, Daniel},
     TITLE = {Coalgebras, comodules, pseudocompact algebras and tame
              comodule type},
   JOURNAL = {Colloq. Math.},
  FJOURNAL = {Colloquium Mathematicum},
    VOLUME = {90},
      YEAR = {2001},
    NUMBER = {1},
     PAGES = {101--150},
      ISSN = {0010-1354,1730-6302},
   MRCLASS = {16W30 (16G99)},
  MRNUMBER = {1874368},
MRREVIEWER = {Iain\ G.\ Gordon},
       DOI = {10.4064/cm90-1-9},
       URL = {https://doi.org/10.4064/cm90-1-9},
}

@article {ChenRingel,
    AUTHOR = {Chen, Xiao-Wu and Ringel, Claus Michael},
     TITLE = {Hereditary triangulated categories},
   JOURNAL = {J. Noncommut. Geom.},
  FJOURNAL = {Journal of Noncommutative Geometry},
    VOLUME = {12},
      YEAR = {2018},
    NUMBER = {4},
     PAGES = {1425--1444},
      ISSN = {1661-6952,1661-6960},
   MRCLASS = {18E30 (13D09 18E10)},
  MRNUMBER = {3896230},
MRREVIEWER = {Dejun\ Wu},
       DOI = {10.4171/JNCG/311},
       URL = {https://doi.org/10.4171/JNCG/311},
}

@article {AlNofayee,
    AUTHOR = {Al-Nofayee, Salah},
     TITLE = {Simple objects in the heart of a t-structure},
   JOURNAL = {J. Pure Appl. Algebra},
  FJOURNAL = {Journal of Pure and Applied Algebra},
    VOLUME = {213},
      YEAR = {2009},
    NUMBER = {1},
     PAGES = {54--59},
      ISSN = {0022-4049,1873-1376},
   MRCLASS = {18E30 (16G10)},
  MRNUMBER = {2462984},
MRREVIEWER = {Paul\ Balmer},
       DOI = {10.1016/j.jpaa.2008.05.014},
       URL = {https://doi.org/10.1016/j.jpaa.2008.05.014},
}

@article {KoenigYang,
    AUTHOR = {Koenig, Steffen and Yang, Dong},
     TITLE = {Silting objects, simple-minded collections, t-structures
              and co-t-structures for finite-dimensional algebras},
   JOURNAL = {Doc. Math.},
  FJOURNAL = {Documenta Mathematica},
    VOLUME = {19},
      YEAR = {2014},
     PAGES = {403--438},
      ISSN = {1431-0635,1431-0643},
   MRCLASS = {16E35 (18E30)},
  MRNUMBER = {3178243},
MRREVIEWER = {Octavio\ Mendoza Hern\'{a}ndez},
}

@article{Hubery,
  title={Characterising the bounded derived category of an hereditary abelian category},
  author={Andrew Hubery},
  journal={arXiv:1612.06674 [math.RT]},
  year={2016},
  url={https://arxiv.org/abs/1612.06674}
}

@article {Gabriel,
    AUTHOR = {Gabriel, Pierre},
     TITLE = {Des cat\'{e}gories ab\'{e}liennes},
   JOURNAL = {Bull. Soc. Math. France},
  FJOURNAL = {Bulletin de la Soci\'{e}t\'{e} Math\'{e}matique de France},
    VOLUME = {90},
      YEAR = {1962},
     PAGES = {323--448},
      ISSN = {0037-9484},
   MRCLASS = {18.20},
  MRNUMBER = {232821},
MRREVIEWER = {T.-Y.\ Lam},
       URL = {http://www.numdam.org/item?id=BSMF_1962__90__323_0},
}

\end{document}